\documentclass[11pt]{article}

\usepackage[utf8]{inputenc} 
\usepackage[T1]{fontenc}
\usepackage{amsmath, amssymb, amsthm}
\usepackage{bm,bbm}
\usepackage{mathtools}
\usepackage{commath}
\usepackage{graphicx}
\usepackage[font=small,labelfont=bf]{caption}
\usepackage{subcaption}
\usepackage{array,multirow,booktabs}
\usepackage{makecell} % Tabular column heads and multilined cells
\usepackage{tabularx} % Tabulars with adjustable-width columns
\usepackage{rotating} % Rotation tools, including rotated full-page floats

\usepackage{indentfirst}
\usepackage{hyperref}
\usepackage{enumerate}
\usepackage{makecell}
\usepackage{fancyhdr}
\usepackage{setspace}
\usepackage{pdflscape}
\usepackage[margin=2.5cm]{geometry}
\usepackage{tabularx}
\usepackage[title]{appendix}

\newtheorem{theorem}{Theorem}[section]
\newtheorem{lemma}[theorem]{Lemma}
\newtheorem{proposition}[theorem]{Proposition}

\theoremstyle{definition}

\newtheorem{example}{Example}[section]
\newtheorem{remark}{Remark}[section]

\newcommand{\1}{\mathbbm{1}}

\newcommand{\mub}{\boldsymbol{\mu}}

\newcommand{\E}{\mathbb{E}}
\newcommand{\N}{\mathbb{N}}
\newcommand{\R}{\mathbb{R}}

\newcommand{\Mcal}{\mathcal{M}}
\newcommand{\Fcal}{\mathcal{F}}
\newcommand{\Gcal}{\mathcal{G}}
\newcommand{\Hcal}{\mathcal{H}}

\newcommand{\Pcal}{\mathcal{P}}
\newcommand{\Scal}{\mathcal{S}}
\newcommand{\Xcal}{\mathcal{X}}
\newcommand{\Ycal}{\mathcal{Y}}

\newcommand{\FR}{\mathrm{FR}}
\newcommand{\KL}{\mathrm{KL}}
\newcommand{\Hell}{\mathrm{H}}

\newcommand{\xv}{\pmb{x}}
\newcommand{\yv}{\pmb{y}}

\renewcommand{\tilde}{\widetilde}

\DeclareMathOperator{\arccosh}{arccosh}
\DeclareMathOperator{\arctanh}{arctanh}
\DeclareMathOperator{\supp}{supp}

\DeclareMathOperator{\tr}{tr}
\DeclareMathOperator{\vect}{vec}

\renewcommand{\d}{\mathrm{d}}
\renewcommand{\t}{\mathsf{T}}

\newcommand{\suchthat}{\;\middle\vert\;}

\numberwithin{equation}{section}

\makeatletter
\let\@fnsymbol\@arabic
\makeatother
\newcommand*\samethanks[1][\value{footnote}]{\footnotemark[#1]}

\title{\textbf{On Closed-Form Expressions for the Fisher--Rao Distance}}
\author{
    Henrique~K.~Miyamoto\thanks{
		Université Paris-Saclay, CentraleSupélec, CNRS, France. E-mail: \texttt{henrique.miyamoto@centralesupelec.fr}.
	}\hspace{1em}
    Fábio~C.~C.~Meneghetti\thanks{
		IMECC, Unicamp, Brazil. E-mail: \texttt{sueli@unicamp.br}.
	}\hspace{1em}
    Julianna~Pinele\thanks{
        CETEC, UFRB, Brazil. E-mail: \texttt{julianna.pinele@ufrb.edu.br}
    }\hspace{1em}
	Sueli~I.~R.~Costa\samethanks[2]
}
\date{\vspace{-1.5em}}

\begin{document}

\maketitle

\begin{abstract}
	The Fisher--Rao distance is the geodesic distance between probability distributions in a statistical manifold equipped with the Fisher metric, which is a natural choice of Riemannian metric on such manifolds. It has recently been applied to supervised and unsupervised problems in machine learning, in various contexts. Finding closed-form expressions for the Fisher--Rao distance is generally a non-trivial task, and those are only available for a few families of probability distributions. In this survey, we collect examples of closed-form expressions for the Fisher--Rao distance of both discrete and continuous distributions, aiming to present them in a unified and accessible language. In doing so, we also: illustrate the relation between negative multinomial distributions and the hyperbolic model, include a few new examples, and write a few more in the standard form of elliptical distributions.
\end{abstract}

\tableofcontents

\section{Introduction}

Information geometry~\cite{amari2000,calin2014,nielsen2020} uses the tools of differential geometry to study spaces of probability distributions by regarding them as differential manifolds, called \emph{statistical manifolds}. When these distributions are parametric, a structure of interest is the \emph{Fisher metric}, a Riemannian metric induced by the Fisher information matrix. This is essentially the unique Riemannian metric on statistical manifolds that is invariant by sufficient statistics~\cite{ay2015,le2017}, making it a natural choice to study the geometry of these manifolds. Moreover, this structure allows one to define the \emph{Fisher--Rao distance} between two probability distributions on the same statistical manifold as the geodesic distance between them, i.e., the length of the minimising path, according to the Fisher metric.

The idea of considering the geodesic distance in Riemannian manifolds equipped with the Fisher metric was first suggested by Hotelling in 1930~\cite{hotelling1930} (reprinted in \cite{stigler2007}), and later in 1945 in a landmark paper by Rao~\cite{rao1945} (reprinted in \cite{rao1992}). This has influenced many authors to study the Fisher--Rao distance in different families of probability distributions in the following years~\cite{atkinson1981,ayadi2023,burbea1986,calvo1991,li2022,micchelli2005,minarro1993,mithcell1988,oller1987,oller1985,pinele2020,rao1987,verdoolaege2012,villarroya1993}. It is important to note that, contrary to commonly used divergence measures, such as the Kullback-Leibler divergence, the Fisher--Rao distance is a proper distance, i.e., it is symmetric and the triangle inequality holds---properties that could be required when comparing two distributions, depending on the application.

More recently, the Fisher--Rao distance has gained attention especially in applications to machine learning problems. In the context of unsupervised learning, it has been used for clustering different types of data: shape clustering applied to morphometry~\cite{gattone2017}, clustering of financial returns~\cite{taylor2019}, image segmentation~\cite{pinele2020} and identification of diseases from medical data~\cite{lebrigant2021-beta,rebbah2019}. When it comes to supervised learning, it has been used to analyse the geometry in the latent space of generative models~\cite{arvanitidis2022}, to enhance robustness against adversarial attacks~\cite{picot2023,shi-garrier2024}, in detection of out-of-distribution samples~\cite{gomes2022}, as a loss function for learning under label noise~\cite{miyamoto2023}, and to classify EEG signals of brain-computer interfaces~\cite{bouchard2023}.

However, finding closed-form expressions for the Fisher--Rao distance of arbitrary distributions is not a trivial task---as a matter of fact, more generally, finding geodesics in an arbitrary manifold is a difficult problem in differential geometry. Having that in mind, in this work, we collect examples of statistical models for which closed-form expressions for the Fisher--Rao distance are available. Most of them have been published over the last decades, in different places, and we aim to present these results in a unified and accessible language, hoping to bring them to a broader audience. In curating this collection, we also add a few contributions, such as: illustrating the relation between the manifold of negative multinomial distributions and the hyperbolic model; including a few new examples (Rayleigh, Erlang, Laplace, generalised Gaussian, power function, inverse Wishart), to the best of the authors' knowledge; and writing more examples in the standard form of univariate elliptical distributions (Laplace, generalised Gaussian, logistic). Finally, we note that numerical methods have been proposed to compute the Fisher--Rao distance when no closed-form expression is available, as in~\cite{arwini2008,han2014,lebrigant2022,lebrigant2021-beta,nielsen2023,rebbah2019,reverter2003}, but those techniques are beyond the scope of the present work.

We review preliminaries of information geometry in Section~\ref{sec:preliminaries} and of hyperbolic geometry in Section~\ref{sec:hyperbolic-geometry}. We then collect closed-form expressions for discrete distributions in Section~\ref{sec:discrete-distributions}, and for continuous distributions in Section~\ref{sec:continuous-distributions}. Product distributions are discussed in Section~\ref{sec:product-distributions}, and Section~\ref{sec:conclusion} concludes the paper.

\paragraph{Notation.} We denote the sets $\N = \{ 0, 1, 2, \dots \}$, $\N^* = \{ 1, 2, 3, \dots \}$, $\R_+ = \left[ 0, \infty \right[$, and $\R_+^* = \left] 0, \infty \right[$. $\1_A(x)$ is the indicator function, that takes value $1$ if $x \in A$, and $0$ otherwise. $\delta_{ij} \coloneqq \1_{\{j\}}(i)$ denotes the Kronecker delta. We denote $\dot{x}(t) \coloneqq \frac{\d}{\d t}x(t)$. $P_n(\R)$ denotes the cone of $n \times n$ real symmetric definite-positive matrices.

Let $(\Omega, \Gcal, P)$ be a probability space and $X \colon \Omega \to \Xcal$ a random variable in the $\sigma$-finite measure space $(\Xcal, \Fcal, \mu)$. The push-forward measure of $P$ by $X$ is given by $X_*P(E) \coloneqq P(X^{-1}(E))$, for $E \in \Fcal$, and we assume that $X_*P$ is absolutely continuous with respect to $\mu$. The Radon-Nikodym derivative $p \coloneqq \frac{\d X_*P}{\d \mu} \colon \Xcal \to \R$ can be seen as the probability mass or density function (p.m.f. or p.d.f.), respectively, in the cases that $\Xcal$ is discrete or continuous. When $\Xcal$ is discrete, we take $\mu$ as the counting measure, and the integral with respect to $\mu$ becomes a summation; when $\Xcal = \R^n$, we take $\mu$ as the Lebesgue measure.

A \emph{statistical model}~\cite{amari2000} 
\begin{equation} \label{eq:statistical-model}
	\Scal \coloneqq \left\{ p_{\xi} = p(x; \xi) \suchthat \xi = (\xi^1, \dots, \xi^n) \in \Xi \subseteq \R^n \right\}
\end{equation}
is a family of probability distributions $p_\xi$ parametrised by $n$-dimensional vectors $\xi = \left( \xi^1, \dots, \xi^n \right)$ such that the mapping $\xi \mapsto p_\xi$ is injective and $\Xi$ is an open set of $\R^n$. We consider statistical models in which the support of $p_\xi$ does not depend on $\xi$ and we take $\Xcal = \supp p_\xi$, unless otherwise stated. Note that $\Scal$ is contained in the infinite-dimensional space $\Pcal(\Xcal) \coloneqq \left\{ p \in L^1(\mu) \suchthat p > 0, \ \int_{\Xcal} p~\d\mu =1 \right\}$ of positive, $\mu$-integrable functions of unit total measure.

\begin{figure}
	\centering
	\includegraphics[width=0.45\linewidth]{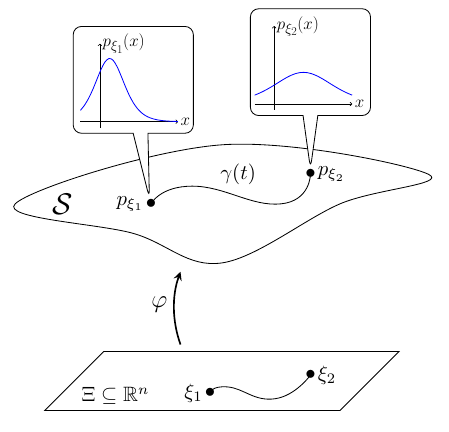}
	\caption{Schematic representing the parametrisation~$\varphi$ from the parameter space~$\Xi$ to the statistical manifold~$\Scal$. The curve~$\gamma(t)$ joins two points in the manifold, which are probability density (or mass) functions.}
	\label{fig:manifold}
\end{figure}

To introduce a differentiable structure in $\Scal$, we consider the following assumptions: 1)~the parametrisation $\varphi \colon \Xi \to \Pcal(\Xcal),\ \varphi(\xi) = p_\xi$ is a homeomorphism on its image; 2)~denoting $\partial_i \coloneqq \frac{\partial}{\partial \xi^i}$, the functions $\left\{ \partial_1 p_\xi, \dots, \partial_n p_\xi \right\}$ are linearly independent; 3)~the mapping $\xi \mapsto p_\xi(x)$ is smooth, for all $x \in \Xcal$; 4)~the partial derivatives $\partial_i p_\xi(x)$ commute with the integrals. Moreover, by considering diffeomorphic parametrisations as equivalent, $\Scal$ becomes a differentiable manifold, that we may call a \emph{statistical manifold}\footnote{
	We follow the nomenclature from~\cite{amari2000}, but remark that, more generally, \emph{statistical manifold} refers to a manifold equipped with a Riemannian metric and a $3$-symmetric tensor, from a purely geometric point of view~\cite{lauritzen1987} (see also~\cite[\S~4.5]{ay2017}). We will restrict ourselves to the case of parametric statistical models presented above. 
}. Note that the parametrisation $\xi \mapsto p_\xi$ is a global coordinate system for this manifold (see Figure~\ref{fig:manifold}).

We can further equip the statistical manifold $\Scal$ with a Riemannian metric. Denoting $\ell(\xi) \coloneqq \log p_\xi$ the log-likelihood function, the elements of the \emph{Fisher information matrix} (or simply \emph{Fisher matrix}) $G(\xi) = \left[ g_{ij}(\xi) \right]_{i,j}$ are defined as
\begin{equation} \label{eq:fisher-matrix-gij}
	g_{ij} \coloneqq g_{ij}(\xi) \coloneqq \E \left[ \partial_i \ell(\xi) \partial_j \ell(\xi) \right],
\end{equation}

\noindent for $1 \le i,j \le n$, where the expectation is taken with respect to $p_\xi$. Explicitly,
\begin{equation*}
	g_{ij}(\xi) = \int_\Xcal p_\xi \left(\frac{\partial}{\partial \xi^i}\log p_\xi\right) \left(\frac{\partial}{\partial \xi^j}\log p_\xi\right)\d\mu.
\end{equation*}

Alternatively, the Fisher matrix can be written as the negative expectation of the Hessian of the log-likelihood function, a result that can make the computation of the Fisher matrix easier in some cases:

\begin{proposition}[{\cite[Prop.~1.6.3]{calin2014}}] \label{prop:alternative-form}
	The elements of the Fisher matrix can be expressed as
	\begin{equation}\label{eq:fisher-matrix-hessian}
		g_{ij}(\xi) = -\E \left[ \partial_j \partial_i \ell(\xi) \right].
	\end{equation}
\end{proposition}
\begin{proof}
	As $\int_{\Xcal} p_\xi~\d\mu = 1$, taking the derivative yields $\int_{\Xcal}p_\xi \partial_i \log p_\xi~\d\mu = \int_{\Xcal} \partial_i p_\xi~\d\mu = 0$. By taking the derivative again, we have
	\begin{align*}
		0 = \partial_j \int_{\Xcal} p_\xi \partial_i \log p_\xi~\d\mu
		&= \int_{\Xcal} \partial_j p_\xi \partial_i \log p_\xi~\d\mu+ \int_{\Xcal} p_\xi \partial_j \partial_i \log p_\xi~\d\mu\\
		&= \int_{\Xcal} p_\xi (\partial_j \log p_\xi) (\partial_i \log p_\xi)~\d\mu+ \int_{\Xcal} p_\xi \partial_j \partial_i \log p_\xi~\d\mu \\
		&= \E \left[(\partial_j \log p_\xi) (\partial_i \log p_\xi)\right] + \E \left[\partial_j\partial_i \log p_\xi\right].
	\end{align*}
\end{proof}

Since the Fisher matrix is symmetric and positive-definite, it defines a Riemannian metric~$g_{p_{\xi}}$ (also denoted simply $g_{\xi}$), called the \emph{Fisher metric}; that is, a family of inner products $g_{\xi} \colon T_{p_\xi}\Scal \times T_{p_\xi}\Scal \to \R$ that vary smoothly on the statistical manifold. Applying the Fisher metric to two vectors $v_1 =\d\varphi_{p_\xi}(\xi_1)$ and $v_2 = \d\varphi_{p_\xi}(\xi_2)$ in the tangent space $T_{p_{\xi}}\Scal$ is equivalent to computing an inner product mediated by the Fisher matrix $G(\xi)$ between the respective local-coordinate vectors $\xi_1, \xi_2 \in \R^n$:
\begin{equation} \label{eq:inner-product}
	\langle v_1, v_2 \rangle_{G(\xi)}
	\coloneqq g_\xi \left( v_1, v_2 \right)
	\coloneqq g_{p_\xi}\left( v_1, v_2 \right)
	= \xi_1^\t \, G(\xi) \, \xi_2.
\end{equation}

The following results will help some derivations in the rest of the text. First, we note that, as any Riemannian metric, the Fisher metric is covariant under reparametrisation of the parameter space:

\begin{proposition}[{\cite[Thm.~1.6.5]{calin2014}}] \label{prop:reparametrisation-parameter-space}
	The Fisher matrix is covariant under reparametrisation of the parameters space, that is, given two coordinate systems $\xi = (\xi^1, \dots, \xi^n)$ and $\theta = (\theta^1, \dots, \theta^n)$, related by the bijection $\xi = \xi(\theta)$, the Fisher matrix transforms its coordinates as
	\begin{equation} \label{eq:covariant-parameters}
		\tilde{G}(\theta) = \left[ \frac{\d \xi}{\d \theta} \right]^\t G\left(\xi(\theta)\right) \left[ \frac{\d \xi}{\d \theta} \right],
	\end{equation}
	where $\left[ \frac{\d \xi}{\d \theta} \right]$ denotes the Jacobian matrix of the transformation $\theta \mapsto  \xi$.
\end{proposition}
\begin{proof}
	Denote $\tilde{p}_\theta \coloneqq p_{\xi(\theta)} = \varphi(\xi(\theta))$ and note that, by the chain rule, $\frac{\partial \tilde{p}_\theta}{\partial \theta^i} = \sum_{k=1}^{n} \frac{\partial \xi^k}{\partial \theta^i} \frac{\partial p_\xi}{\partial \xi^k}$ and $\frac{\partial \tilde{p}_\theta}{\partial \theta^j} = \sum_{r=1}^{n} \frac{\partial \xi^r}{\partial \theta^j} \frac{\partial p_\xi}{\partial \xi^r}$. Thus
	\begin{align*}
		\tilde{g}_{ij}(\theta)
		&= \int_\Xcal \tilde{p}_\theta \left( \frac{\partial}{\partial \theta^i}\log \tilde{p}_\theta \right) \left( \frac{\partial}{\partial \theta^j}\log \tilde{p}_\theta\right) \d\mu
		= \int_\Xcal \frac{1}{\tilde{p}_\theta}  \frac{\partial \tilde{p}_\theta}{\partial \theta^i} \frac{\partial \tilde{p}_\theta}{\partial \theta^j}~\d\mu\\
		&= \sum_{k=1}^{n}\sum_{r=1}^{n} \left( \int_\Xcal \frac{1}{p_{\xi(\theta)}} \frac{\partial p_\xi}{\partial \xi^k}  \frac{\partial p_\xi}{\partial \xi^r}~\d\mu \right) \frac{\partial \xi^k}{\partial \theta^i} \frac{\partial \xi^r}{\partial \theta^j}
		= \sum_{k=1}^{n}\sum_{r=1}^{n} g_{kr}(\xi(\theta)) \frac{\partial \xi^k}{\partial \theta^i} \frac{\partial \xi^r}{\partial \theta^j}.
	\end{align*}
\end{proof}

Second, we state an invariance property specific to the Fisher metric:

\begin{proposition}[{\cite[Thm.~1.6.4]{calin2014}}] \label{prop:reparametrisation-sample-space}
	Let $X \colon \Omega \to \Xcal \subseteq \R^n$ be a random variable distributed according to~$p_\xi$. The Fisher metric is invariant under reparametrisations of the sample space~$\Xcal$.
\end{proposition}
\begin{proof}
	Consider the reparametrisation by the bijection $f \colon \Xcal \to \Ycal \subseteq \R^n$ and denote $\tilde{p}_\xi$ the distribution associated to the random variable $Y \coloneqq f(X)$. The Jacobian determinant $\left|\frac{\d f}{\d x}\right|$ of the transformation $f$ relates the relation between the densities:
	\begin{equation}
		p_\xi(x) = \tilde{p}_\xi(y) \left| \frac{\d f}{\d x} \right|.
	\end{equation}
	
	\noindent The log-likelihood functions are $\tilde\ell(\xi) = \log \tilde{p}_\xi (y) = \log \tilde{p}_\xi (f(x))$ and $\ell(\xi) = \log p_\xi(x) = \log \tilde{p}_\xi(y) + \log \left| \frac{\d f}{\d x} \right|$. As $f$ does not depend on the parameter~$\xi$, we have $\partial_i \ell(\xi) = \partial_i \tilde{\ell}(\xi)$, whence
	
	\begin{align*}
		g_{ij}(\xi)
		&= \int_\Xcal p_\xi \partial_i \ell(\xi) \partial_i \ell(\xi)~\d\mu
		= \int_\Xcal (\tilde{p}_\xi \circ f) \left| \frac{\d f}{\d x} \right| \partial_i \tilde\ell(\xi) \partial_i \tilde\ell(\xi)~\d\mu\\
		&= \int_\Ycal \tilde{p}_\xi \partial_i \tilde\ell(\xi) \partial_i \tilde\ell(\xi)~\d(f_{*}\mu)
		= \tilde{g}_{ij}(\xi).
	\end{align*}
\end{proof}

Furthermore, it can be shown that the Fisher metric is the unique Riemannian metric (up to a multiplicative constant) in statistical manifolds that is invariant under sufficient statistics~\cite[Thm.~1.2]{ay2015} (see also~\cite{le2017}). This invariance characterisation justifies the choice of this metric to study the geometry of statistical models.

The Fisher metric induces a notion of distance between two distributions in the same statistical manifold, called \emph{Fisher--Rao distance}, and given by the geodesic distance between these points. Specifically, consider a curve $\xi \colon \left[0,1\right] \to \Xi$ in the parameter space and its image $\gamma \colon \left[0,1\right] \to \Scal$ by the parametrisation $\varphi$, i.e., $\gamma(t) = (\varphi \circ \xi) (t)$, for $t \in \left[0,1\right]$. Since $\dot{\gamma}(t) = \d\varphi_{\xi(t)}\big(\dot{\xi}(t)\big)$, in the Fisher geometry, the length of $\gamma$ can be computed as
\begin{equation} \label{eq:curve-length}
	l(\gamma)
	\coloneqq \int_0^1 \sqrt{ \left\langle \dot{\gamma}(t), \dot{\gamma}(t) \right\rangle_{G \left( \xi (t) \right)} }~\d t
	= \int_0^1  \sqrt{ \dot{\xi}(t)^\t \, G(\xi(t)) \, \dot{\xi}(t) }~\d t.
\end{equation}

\noindent Then, given two distributions $p_{\xi_1}$ and $p_{\xi_2}$ in $\Scal$, the Fisher--Rao distance between them\footnote{
	We shall also refer to the Fisher--Rao distance between two parametric distributions as the distance between their respective parameters.
} is the infimum of the length of piecewise differentiable curves~$\gamma$ joining these two points:
\begin{equation} \label{eq:fisher-rao-distance}
	d_{\FR}(\xi_1, \xi_2)
	\coloneqq d_{\FR}(p_{\xi_1}, p_{\xi_2})
	\coloneqq \inf_{\gamma} \left\{ l(\gamma) \suchthat \gamma(0) = p_{\xi_1},\ \gamma(1) = p_{\xi_2} \right\}.
\end{equation}

A curve $\gamma(t) = (\varphi \circ \xi) (t)$ is a \emph{geodesic} if, in local coordinates $\xi=(\xi^1, \dots, \xi^n)$, the curve $\xi(t)$ is a solution to the geodesic differential equations
\begin{equation} \label{eq:geodesic-diff-eq}
	\ddot{\xi}^k(t) + \sum_{i=1}^{n} \sum_{j=1}^{n} \Gamma_{ij}^{k} \dot{\xi}^i(t) \dot{\xi}^j(t) = 0, \quad k \in \{ 1, \dots, n \},
\end{equation}
where $\Gamma_{ij}^{k}$ are the Christoffel symbols of the second kind, which can be obtained from the equations
\begin{equation*}
	\sum_{k=1}^{n} g_{\ell k} \Gamma_{ij}^{k} = \frac{1}{2} \left( \frac{\partial}{\partial \xi^i} g_{j\ell} + \frac{\partial}{\partial \xi^j} g_{\ell i} - \frac{\partial}{\partial \xi^\ell} g_{ij} \right), \quad i,j,\ell \in \{1,\dots,n\}
\end{equation*}

\noindent The Hopf--Rinow theorem (e.g., \cite[Thm.~6.4.6]{klingenber1978}) provides a sufficient condition for the minimum length, as in~\eqref{eq:fisher-rao-distance}, to be realised by a geodesic: if $(\Scal, d_\FR)$ is connected and complete as a metric space, then any two points $p,q$ in the manifold $\Scal$ can be joined by a minimising curve which is a geodesic, that is, a curve whose length is equal to the Fisher--Rao distance $d_{\FR}(p,q)$. This condition is satisfied in all statistical manifolds considered in this paper. 

\begin{remark}
	The Fisher--Rao distance is related to the Kullback-Leibler divergence~\cite[Thm.~4.4.5]{calin2014}:
	\begin{equation*}
		D_{\KL}(p_{\xi_1} \| p_{\xi_2}) = \frac{1}{2} d_{\FR}^2(p_{\xi_1},p_{\xi_2}) + o\left(d_{\FR}^2(p_{\xi_1},p_{\xi_2})\right),
	\end{equation*}
	where $o(x)$ represents a quantity such that $\lim_{x\to 0} \frac{o(x)}{x} = 0$. This means that the Kullback-Leibler divergence locally behaves as a `squared distance'. Differently from general statistical divergences, the Fisher--Rao distance is a proper distance, i.e., it is symmetric and th`e triangle inequality holds.
\end{remark}

Unfortunately, finding the Fisher--Rao distance between two distributions in a statistical manifold usually is a non-trivial task, since it involves finding the minimising geodesics, potentially by solving the geodesics differential equations~\eqref{eq:geodesic-diff-eq}, and then evaluating the integral in~\eqref{eq:curve-length}. In the case of one-dimensional manifolds, i.e., those which are parametrised by a single real number, computing the Fisher--Rao distance is easier, since the geodesics are immediately given. In such cases, the Fisher matrix $G(\xi) = [g_{11}(\xi)]$ contains a single element (also called \emph{Fisher information}). Given two parameters $\xi_1$ and $\xi_2$, there is only one path joining them, whose length does not depend on the chosen parametrisation. In particular, we can consider the arc length parametrisation $\xi(t) = t$, with $t \in \left[ \xi_1, \xi_2 \right]$, so that $\big|\dot{\xi}(t)\big| = 1$. Thus the expression for the length of the curve $\gamma(t) = (\varphi \circ \xi)(t)$ in~\eqref{eq:curve-length} becomes
\begin{equation*}
	l(\gamma)
	= \int_{\xi_1}^{\xi_2} \sqrt{g_{11}(\xi(t))}~\d t
	= \int_{\xi_1}^{\xi_2} \sqrt{g_{11}(\xi)}~\d \xi,
\end{equation*}

\noindent and the Fisher--Rao distance between distributions parametrised by $\xi_1$ and $\xi_2$ is
\begin{equation} \label{eq:Fisher--Rao-distance-dim1}
	d_{\FR}(\xi_1,\xi_2)
	= \left| \int_{\xi_1}^{\xi_2} \sqrt{g_{11}(\xi)}~\d \xi \right|.
\end{equation}

\noindent For higher-dimensional manifolds, the techniques to find the geodesics and the Fisher--Rao distance consist in directly solving the geodesic differential equations, or in doing an analogy with some well-known geometry (e.g., spherical, hyperbolic), as we shall see in the next sections.

\begin{remark}
	The expression \eqref{eq:Fisher--Rao-distance-dim1} allows one to find the Fisher--Rao distance in one-dimensional submanifolds of higher dimensional statistical manifolds, i.e., when only one parameter is allowed to vary and the others are fixed. This has been done for Gamma, Weibull and power function distributions in~\cite{atkinson1981,burbea1986}.
\end{remark}

\section{Hyperbolic Geometry Results} \label{sec:hyperbolic-geometry}

We recall in this section some classical results from hyperbolic geometry~\cite{beardon1983,cannon1997,ratcliffe2006}, since many of the statistical manifolds studied in this work are related to that geometry. These will be extensively used particularly in \S~\ref{subsec:negative-multinomial} and \S~\ref{subsec:elliptical}---\ref{subsec:power-function}. If desired, the reader may skip this section for now, and return to it when reading those subsections to get the details of the derivations.

We start with the hyperbolic geometry in dimension two, to be used in the approach of the statistical manifolds in \S~\ref{subsec:elliptical}---\ref{subsec:power-function}, analogously to what is done in~\cite{costa2015}. In this case, we consider the \emph{Poincaré half-plane} $\mathcal{H}^2 \coloneqq \left\{ (x,y) \in \R^2 \suchthat y>0 \right\}$ as a model for hyperbolic geometry, with the metric given in matrix form by

\begin{equation} \label{eq:hyperbolic-metric}
	G_{\Hcal^2}(x,y) =
	\begin{pmatrix}
		\frac{1}{y^2} & 0\\ 0 & \frac{1}{y^2}\\
	\end{pmatrix}.
\end{equation}

\noindent The geodesics in this manifold are vertical half-lines and half-circles centred at $y=0$, and the geodesic distance between two points is given by the following equivalent expressions:
\begin{align*}
	d_{\Hcal^2}\left((x_1,y_1),(x_2,y_2)\right)
	&= \log \left(\frac{
		\sqrt{(x_1-x_2)^2 + (y_1+y_2)^2} + \sqrt{(x_1-x_2)^2 + (y_1-y_2)^2}
	}{
		\sqrt{(x_1-x_2)^2 + (y_1+y_2)^2} - \sqrt{(x_1-x_2)^2 + (y_1-y_2)^2}
	}\right)\\
	&= \arccosh \left( 1 + \frac{(x_1-x_2)^2 + (y_1-y_2)^2}{2y_1y_2} \right)\\
	&= 2 \arctanh \left(\sqrt{\frac{
			(x_1-x_2)^2 + (y_1-y_2)^2
		}{
			(x_1-x_2)^2 + (y_1+y_2)^2
	}}\right).
\end{align*}

The next lemma allows one to relate the geodesic distance in a class of two-dimensional Riemannian manifold to that in the Poincaré half-plane.

\begin{lemma}\label{lemma:poincare-isometry}
	Consider the Poincaré half-plane model~$\Hcal^2$ with the hyperbolic metric~\eqref{eq:hyperbolic-metric}, and a two-dimensional Riemannian manifold~$(\Mcal,g_\Mcal)$, parametrised by a global coordinate system $\varphi \colon \Pi \subseteq \R \times \R_+^* \to \Mcal$. If the metric $g_{\Mcal}$ is given in matrix form by
	\begin{equation*}
		G_{\mathcal{M}}(x,y) =
		\begin{pmatrix}
			\frac{a}{y^2} & 0\\ 0 & \frac{b}{y^2}\\
		\end{pmatrix},
	\end{equation*}
	with $a, b$ positive constants, then the geodesic distance between two points in $\Mcal$ is
	\begin{equation} \label{eq:fr-poincare-relation}
		d_{\mathcal{M}} \big( \varphi(x_1,y_1), \varphi(x_2,y_2) \big)
		= \sqrt{b} \ d_{\Hcal^2} \left( (\sqrt{a}x_1, \sqrt{b}y_1), (\sqrt{a}x_2,\sqrt{b}y_2) \right).  
	\end{equation}
\end{lemma}
\begin{proof}
	Consider a curve $\pi(t) = \big(x(t), y(t)\big)$ in the parameter space~$\Pi$ and its image $\alpha(t) = (\varphi \circ \pi)(t)$ by the parametrisation. Applying the metric~$G_{\Mcal}(x,y)$ to $\dot{\alpha}(t) = \mathrm{d}\varphi_{\pi(t)}\big(\dot{\pi}(t)\big)$ in $\mathcal{M}$ gives
	\begin{align*}
		\left\langle \dot{\alpha}(t), \dot{\alpha}(t) \right\rangle_{G(\pi(t))}
		&= \begin{pmatrix}
			\dot{x}(t) & \dot{y}(t)
		\end{pmatrix}
		\begin{pmatrix}
			\frac{a}{(y(t))^2} & 0\\ 0 & \frac{b}{(y(t))^2}\\
		\end{pmatrix}
		\begin{pmatrix}
			\dot{x}(t) \\ \dot{y}(t)
		\end{pmatrix}\\
		&= \frac{1}{(y(t))^2} \left( a (\dot{x}(t))^2 + b (\dot{y}(t))^2 \right).
	\end{align*}
	
	\noindent On the other hand, consider the diffeomorphism $\psi \colon \Mcal \to \Hcal^2$, $\varphi(x,y) \mapsto (\sqrt{a}x,\sqrt{b}y)$, and the image curve $\beta(t) \coloneqq (\psi \circ \alpha)(t) = (\sqrt{a}x(t),\sqrt{b}y(t))$. Applying $G_{\Hcal^2}(\beta(t))$ to $\dot{\beta}(t) = \d\psi_{\alpha(t)}(\dot{\alpha}(t))$ in~$\Hcal^2$ gives
	\begin{align*}
		\left\langle \dot{\beta}(t), \dot{\beta}(t) \right\rangle_{G_{\Hcal^2}(\beta(t))}
		&= \begin{pmatrix}
			\sqrt{a}\dot{x}(t) & \sqrt{b}\dot{y}(t)
		\end{pmatrix}
		\begin{pmatrix}
			\frac{1}{b(y(t))^2} & 0\\ 0 & \frac{1}{b(y(t))^2}\\
		\end{pmatrix}
		\begin{pmatrix}
			\sqrt{a}\dot{x}(t) \\ \sqrt{b}\dot{y}(t)
		\end{pmatrix}\\
		&= \frac{1}{b(y(t))^2} \left( a (\dot{x}(t))^2 + b (\dot{y}(t))^2 \right).
	\end{align*}
	
	\noindent Thus, $\left\langle \dot{\alpha}(t), \dot{\alpha}(t) \right\rangle_{G_{\mathcal{M}}(\alpha(t))}= b\left\langle \dot{\beta}(t), \dot{\beta}(t) \right\rangle_{G_{\Hcal^2}(\beta(t))}$, that is, $\left\| \dot{\alpha}(t)\right\|_{G_{\Mcal}(\alpha(t))}=\sqrt{b} \left\| \dot{\beta}(t) \right\|_{G_{\Hcal^2}(\beta(t))}$. This implies that a curve $\alpha(t)$ connecting two points in $\mathcal{M}$ is a geodesic (minimises the length) if, and only if, its image $\beta(t)$ is a geodesic in $\Hcal^2$. Taking~\eqref{eq:curve-length} and~\eqref{eq:fisher-rao-distance} into account concludes the proof.
\end{proof}

\begin{remark}
	It is possible to deduce what the geodesics in $\Mcal$ look like in the parameter space~$\Pi$ (i.e., their preimages by $\varphi$), since they are the inverse image by~$\psi$ of geodesics in~$\Hcal^2$. Consider a geodesic in~$\Mcal$ connecting the points $\varphi(x_1,y_1)$ and $\varphi(x_2,y_2)$.	Its preimage in the parameter space is given either by the vertical line joining $y_1$ to $y_2$, if $x_1=x_2$, or, otherwise, by the arc of the half-ellipse joining $(x_1,y_1)$ to $(x_2,y_2)$, centred at $(C,0)$ and given by $\left(\frac{R}{\sqrt{a}}\cos(t)+C,\frac{R}{\sqrt{b}}\sin(t)\right)$, where $C = \frac{a(x_1^2-x_2^2)+b(y_1^2-y_2^2)}{2a(x_1-x_2)}$ and $R = \frac{1}{2} \sqrt{ a\left(x_1-x_2\right)^2 + \frac{b^2}{a}\left(\frac{y_1^2-y_2^2}{x_1-x_2}\right)^2 + 2b\left(y_1^2+y_2^2\right)^2 }$.
\end{remark}

More generally, the $n$-dimensional hyperbolic \emph{half-space model} is the Riemannian manifold $\Hcal^{n} \coloneqq \left\{ (x_1, \dots, x_n) \in \R^n \suchthat x_n > 0 \right\}$, equipped with the metric given in matrix form by
\begin{equation} \label{eq:hyperbolic-metric-general}
	G_{\Hcal^{n}}(x_1,\dots,x_n) =
	\begin{pmatrix}
		\frac{1}{x_n^2} & 0 & \cdots & 0\\
		0 & \frac{1}{x_n^2} & \cdots & 0\\
		\vdots & \vdots & \ddots & \vdots \\
		0 & 0 & \cdots & \frac{1}{x_n^2}
	\end{pmatrix},
\end{equation}

\noindent that is, $\left\langle (u_1,\dots,u_n),(v_1,\dots,v_n) \right\rangle_{G_{\Hcal^n}(x_1,\dots,x_n)} = \frac{1}{x_n^2} \sum_{i=1}^{n} u_iv_i$. This manifold has constant negative curvature. The geodesics in this manifold are vertical half-lines and vertical half-circles centred at the hyperplane $x_n=0$, and the geodesic distance between two points in $\Hcal^n$ is given by
\begin{equation} \label{eq:distance-hyperbolic}
	d_{\Hcal^n}\big( (x_1,\dots,x_n),(y_1,\dots,y_n) \big) = \arccosh \left( 1 + \frac{\sum_{i=1}^{n} (x_i-y_i)^2 }{2x_ny_n} \right).
\end{equation}

Restricted to points in the $(n-1)$-dimensional unit half-sphere in $\Hcal^{n}$, that is, $S_{1}^{n-1} \coloneqq \left\{ (x_1,\dots,x_{n}) \in \Hcal^{n} \suchthat \sum_{i=1}^{n} x_i^2=1 \right\}$, the expression for the distance becomes
\begin{equation} \label{eq:distance-hyperbolic-hemisphere}
	\tilde{d}_{S^{n-1}_1}\big( (x_1,\dots,x_{n}), (y_1,\dots,y_{n}) \big)
	= \arccosh \left( \frac{1 - \sum_{i=1}^{n-1} x_i y_i }{x_n y_n} \right).
\end{equation}
From~\eqref{eq:distance-hyperbolic}, we immediately see that the distance $d_{\Hcal^n}$ is invariant by dilation or contraction, that is, denoting $\xv \coloneqq (x_1,\dots,x_n) \in \Hcal^{n}$ and $\yv \coloneqq (y_1,\dots,y_n) \in \Hcal^{n}$, we have $d_{\Hcal^{n}}(\xv,\yv) = d_{\Hcal^{n}}(\lambda\xv,\lambda\yv)$, for $\lambda\neq0$.
For points $\xv$ and $\yv$ in the half-sphere $S_{r}^{n-1} \coloneqq \left\{ (x_1,\dots,x_n) \in \Hcal^{n} \suchthat \sum_{i=1}^{n} x_i^2 = r^2 \right\}$ of radius~$r$, we can see that the distance between them, restricted to the sphere~$S_r^{n-1}$, is the distance given by~\eqref{eq:distance-hyperbolic-hemisphere} between their projections onto the radius-$1$ sphere, that is,
\begin{equation} \label{eq:distance-hyperbolic-hemisphere-r}
	\tilde{d}_{S_r^{n-1}}\left( \xv, \yv \right)
	= \tilde{d}_{S_1^{n-1}} \left( \frac{\xv}{r}, \frac{\yv}{r} \right).
\end{equation}

The central projection $\pi \colon S_{1}^{n} \subseteq \Hcal^{n+1} \to \Hcal^n$ given by $(x_1,\dots,x_n,x_{n+1}) \mapsto \left( \frac{2x_2}{x_1+1}, \dots, \frac{2x_{n+1}}{x_1+1} \right)$ is an isometry between the unit half-sphere~$S_{1}^{n} \subseteq \Hcal^{n+1}$, with the restriction of the ambient hyperbolic metric, and the hyperbolic space~$\Hcal^{n}$. It provides a  \emph{hemisphere model} in dimension $n+1$ for the $n$-dimensional hyperbolic space~\cite{cannon1997}. The geodesics in the hemisphere model are the inverse image by $f$ of the geodesics in $\Hcal^n$, namely, semicircles orthogonal to the hyperplane $x_{n+1}=0$.

\section{Discrete Distributions} \label{sec:discrete-distributions}

In the following, we begin with examples of one-dimensional statistical manifolds (\S~\ref{subsec:binomial}---\ref{subsec:negative-binomial}), and then consider high-dimensional manifolds (\S~\ref{subsec:categorical}---\ref{subsec:negative-multinomial}). The results are summarised in Table~\ref{tab:discrete-distributions}.

\begin{sidewaystable}
	\centering
	\caption{Fisher--Rao distances for discrete distributions.}
	\label{tab:discrete-distributions}
	\def\arraystretch{2}\tabcolsep=3pt
	\footnotesize
	\begin{tabular}{ccccccc}
		\hline
		& Distribution & Support & Parameters & Fisher matrix & Fisher--Rao distance & References\\ \hline\hline \vspace{0.5em}
		
		Binomial
		& $\binom{n}{x} \theta^x (1-\theta)^{n-x}$
		& $x \in \{ 0,1,\dots,n \}$
		& $\theta \in \left]0,1\right[$
		& $\begin{pmatrix} \frac{n}{\theta(1-\theta)} \end{pmatrix}$
		& $2\sqrt{n}\left| \arcsin\sqrt{\theta_1} - \arcsin\sqrt{\theta_2} \right|$
		& \cite{calin2014,atkinson1981,burbea1986} \\ \vspace{0.5em}
		
		Poisson
		& $\frac{\lambda^xe^{-\lambda}}{x!}$
		& $x \in \N$
		& $\lambda \in \R_+^*$
		& $\begin{pmatrix} \frac{1}{\lambda} \end{pmatrix}$
		& $2 \left| \sqrt{\lambda_1} - \sqrt{\lambda_2} \right|$
		& \cite{calin2014,atkinson1981,burbea1986} \\ \vspace{0.5em}
		
		Geometric
		& $\theta(1-\theta)^{x-1}$
		& $x \in \N^*$
		& $\theta \in \left]0,1\right[$
		& $\begin{pmatrix} \frac{1}{\theta^2(1-\theta)} \end{pmatrix}$
		& $2 \left| \arctanh \sqrt{1-\theta_1} - \arctanh \sqrt{1-\theta_2}  \right|$
		& \cite{calin2014,oller1985} \\ \vspace{0.5em}
		
		\shortstack{Negative\\binomial}
		& $\frac{\Gamma(x+r)}{x! \Gamma(r)} \theta^r (1-\theta)^{x}$
		& $x \in \N$
		& $\theta \in \left]0,1\right[$
		& $\begin{pmatrix} \frac{r}{\theta^2(1-\theta)} \end{pmatrix}$
		& $2 \sqrt{r} \left| \arctanh \sqrt{1-\theta_1} - \arctanh \sqrt{1-\theta_2}  \right| $
		& \cite{burbea1986,oller1985}\\ \vspace{0.5em}
		
		Categorical
		& $\sum_{i=1}^{n} p_i\1_{\{i\}}(x)$
		& $x \in \{1,\dots,n\}$
		&\shortstack{\scriptsize$(p_1,\dots,p_{n-1}) \in \left]0,1\right[^{n-1},$\\ \scriptsize$\textstyle\sum_i p_i < 1$, $p_n\coloneqq 1 - \sum_i p_i$}
		& $ \begin{pmatrix} \frac{\delta_{ij}}{p_i} + \frac{1}{p_n} \end{pmatrix}_{i,j}$
		& $2 \arccos \left( \sum_{i=1}^{n} \sqrt{p_iq_i} \right)$ &\shortstack{\cite{calin2014,ay2015,atkinson1981}\\ \cite{burbea1986,kass1997}} \\ \vspace{0.5em}
		
		Multinomial 
		& $m! \prod_{i=1}^{n} \frac{p_i^{x_i}}{x_i!}$ 
		&\shortstack{\scriptsize$(x_1,\dots,x_n) \in \N^n,$\\ \scriptsize$\textstyle\sum_i x_i = m$}
		&\shortstack{\scriptsize$(p_1,\dots,p_{n-1}) \in \left]0,1\right[^{n-1},$\\ \scriptsize$\textstyle\sum_i p_i < 1$, $p_n\coloneqq 1 - \sum_i p_i$}
		& $\begin{pmatrix} \frac{m\delta_{ij}}{p_i} + \frac{m}{p_n} \end{pmatrix}_{i,j}$ & $2 \sqrt{m} \arccos \left( \sum_{i=1}^{n} \sqrt{p_iq_i} \right)$
		& \cite{calin2014,atkinson1981,burbea1986} \\ \vspace{0.5em}
		
		\shortstack{Negative\\multinomial}
		& $p_n^{x_n}\frac{\Gamma(\sum_{i=1}^{n} x_i)}{\Gamma(x_n)} \prod_{i=1}^{n-1} \frac{p_i^{x_i}}{x_i!}$
		&\shortstack{$(x_1,\dots,x_{n-1})$\\\hspace{2.5em}$\in\N^{n-1}$}
		&\shortstack{\scriptsize$(p_1,\dots,p_{n-1}) \in \left]0,1\right[^{n-1},$\\ \scriptsize$\textstyle\sum_i p_i < 1$, $p_n\coloneqq 1 - \sum_i p_i$}
		& $\begin{pmatrix} \frac{x_n\delta_{ij}}{p_ip_n} + \frac{x_n}{p_n^2} \end{pmatrix}_{i,j}$
		& $2 \sqrt{x_n} \arccosh \left( \frac{1 - \sum_{i=1}^{n-1} \sqrt{p_iq_i}}{\sqrt{p_nq_n}} \right)$
		& \cite{burbea1986,oller1985,khan2022}\\
		\hline
	\end{tabular}
\end{sidewaystable}

\subsection{Binomial} \label{subsec:binomial}

A binomial distribution~\cite{atkinson1981,burbea1986,calin2014} models the probability of having~$x$ successes in~$n$ independent and identically distributed (i.i.d.) Bernoulli experiments with parameter~$\theta$. Its p.m.f. is given by $p(x) = \binom{n}{x} \theta^x (1-\theta)^{n-x}$, defined for $x \in \{0,1,\dots,n\}$, and parametrised by $\theta \in \left]0,1\right[$. In this case, $\partial_\theta \ell(\theta) = \frac{x}{\theta} - \frac{n-x}{1-\theta}$, and the Fisher information is
\begin{align} \label{eq:fisher-matrix-binomial}
	g_{11}(\theta)
	&= \E \left[ \left( \partial_\theta \ell(\theta) \right)^2 \right]
	= \E \left[ \left(\frac{X}{\theta} - \frac{n-X}{1-\theta}\right)^2 \right] \nonumber\\
	&= \frac{\E[X^2]}{\theta^2} + \frac{2\E[X^2]-2n\E[X]}{\theta(1-\theta)} + \frac{n -2n\E[X] + \E[X^2]}{(1-\theta)^2} \nonumber\\
	&= \frac{n}{\theta(1-\theta)},
\end{align}

\noindent where we have used that $\E[X] = n\theta$ and $\E[X^2] = n\theta -n\theta^2 +n^2\theta^2$. The Fisher--Rao distance is then
\begin{equation} \label{eq:fr-distance-binomial}
	d_{\FR}(\theta_1,\theta_2)
	= \left| \int_{\theta_1}^{\theta_2} \sqrt{\frac{n}{\theta(1-\theta)}}~\d \theta \right| 
	= 2\sqrt{n}\left| \arcsin\sqrt{\theta_1} - \arcsin\sqrt{\theta_2} \right|.
\end{equation}

\subsection{Poisson} \label{subsec:poisson}

A Poisson distribution~\cite{atkinson1981,burbea1986,calin2014} has p.m.f. $p(x) = \frac{\lambda^xe^{-\lambda}}{x!}$, defined for $x \in \N$, and parametrised by $\lambda \in \R_+^{*}$. In this case, we have $\partial_{\lambda} \ell(\lambda) = \frac{x}{\lambda} - 1$, and the Fisher information is given by
\begin{align} \label{eq:fisher-matrix-poisson}
	g_{11}(\lambda)
	&= \E \left[ \left(\partial_{\lambda} \ell(\lambda) \right)^2 \right]
	= \E \left[ \left( \frac{X}{\lambda} - 1 \right)^2 \right] \nonumber\\
	&= \frac{1}{\lambda^2} \E [ X^2 ] - \frac{2}{\lambda} \E\left[ X \right] + 1 \nonumber\\
	&= \frac{1}{\lambda},
\end{align}

\noindent where we have used that $\E[X] = \lambda$ and $\E[X^2] = \lambda(\lambda + 1)$. Thus the Fisher--Rao distance is
\begin{equation} \label{eq:fr-distance-poisson}
	d_{\FR}(\lambda_1, \lambda_2)
	= \left| \int_{\lambda_1}^{\lambda_2} \frac{1}{\sqrt{\lambda}}~\d \lambda \right|
	= 2 \left| \sqrt{\lambda_1} - \sqrt{\lambda_2} \right|.
\end{equation}

\subsection{Geometric} \label{subsec:geometric}

A geometric distribution~\cite{calin2014,oller1985} models the number of i.i.d. Bernoulli trials with parameter~$\theta$ needed to obtain one success. Its p.m.f. is $p(x) = \theta(1-\theta)^{x-1}$, defined for $x \in \N^*$, and parametrised by $\theta \in \left]0,1\right[$. We have $\partial_\theta \ell(\theta) = \frac{1}{\theta} - \frac{x-1}{1-\theta}$, and the Fisher information is
\begin{align} \label{eq:fisher-matrix-geometric}
	g_{11}(\theta)
	&= \E \left[ \left(\partial_{\theta} \ell(\theta) \right)^2 \right]
	= \E \left[ \left( \frac{1}{\theta} - \frac{X-1}{1-\theta} \right)^2 \right] \nonumber\\
	&= \frac{1}{\theta^2} + \frac{2-2\E[X]}{\theta(1-\theta)} + \frac{\E[X^2] -2\E[X] + 1}{(1-\theta^2)} \nonumber\\
	&= \frac{1}{\theta^2(1-\theta)},
\end{align}

\noindent where we have used that $\E[X] = \frac{1}{\theta}$ and $\E[X^2] = \frac{2-\theta}{\theta^2}$. The Fisher--Rao distance is 
\begin{equation} \label{eq:fr-distance-geometric}
	d_{\FR}(\theta_1,\theta_2)
	= \left| \int_{\theta_1}^{\theta_2} \frac{1}{\theta\sqrt{1-\theta}}~\d \theta \right| 
	= 2 \left| \arctanh \sqrt{1-\theta_1} - \arctanh \sqrt{1-\theta_2} \right|.
\end{equation}

\subsection{Negative binomial} \label{subsec:negative-binomial}

A negative binomial distribution~\cite{burbea1986,oller1985} models the excess of i.i.d. Bernoulli experiments with parameter~$\theta$ needed until a number of~$r$ successes occur. It has p.m.f. $p(x) = \frac{\Gamma(x+r)}{x! \Gamma(r)} \theta^r (1-\theta)^x$, defined for $x \in \N$, and is parametrised by $\theta \in \left] 0,1\right[$, for a fixed $r$, that can be extend to $r \in \R^*_+$. We have $\partial_\theta \ell(\theta) = \frac{r}{\theta} - \frac{x}{1-\theta}$, and the Fisher information is
\begin{align} \label{eq:fisher-matrix-negative-binomial}
	g_{11}(\theta)
	&= \E \left[ \left(\partial_{\theta} \ell(\theta) \right)^2 \right]
	= \E \left[ \left( \frac{r}{\theta} - \frac{X}{1-\theta} \right)^2 \right]\nonumber\\
	&= \frac{r^2}{\theta^2} - \frac{2r\E[X]}{\theta(1-\theta)} + \frac{\E[X^2]}{(1-\theta^2)}\nonumber\\
	&= \frac{r}{\theta^2(1-\theta)},
\end{align}
where we have used that $\E[X] = \frac{r(1-\theta)}{\theta}$ and $\E[X^2] = \frac{r(1-\theta)+r^2(1-\theta)^2}{(1-\theta)^2}$. The Fisher--Rao distance is then
\begin{equation} \label{eq:fr-distance-negative-binomial}
	d_{\FR}(\theta_1,\theta_2)
	= \left| \int_{\theta_1}^{\theta_2} \frac{\sqrt{r}}{\theta\sqrt{1-\theta}}~\d \theta \right| 
	= 2 \sqrt{r} \left| \arctanh \sqrt{1-\theta_1} - \arctanh \sqrt{1-\theta_2} \right|.
\end{equation}

\subsection{Categorical} \label{subsec:categorical}

A categorical distribution~\cite{atkinson1981,ay2015,burbea1986,kass1997,calin2014} models a random variable taking values in the sample space $\Xcal = \{ 1, 2, \dots, n \}$ with probabilities $p_1, \dots, p_n$, and has p.m.f. $p(x) = \sum_{i=1}^{n} p_i \1_{\{i\}}(x)$. The associated $(n-1)$-dimensional statistical manifold
\begin{equation*}
	\Scal = \left\{ p = \textstyle\sum_{i=1}^{n} p_i \1_{\{i\}} \suchthat p_i \in \left]0,1\right[,\ \textstyle\sum_{i=1}^{n}p_i = 1 \right\}
\end{equation*}

\noindent is in correspondence with the interior of the probability simplex $\mathring{\Delta}^{n-1} \coloneqq \left\{ \pmb{p} = (p_1, \dots, p_n) \suchthat p_i \in \left]0,1\right[,\ \textstyle\sum_{i=1}^{n}p_i = 1 \right\}$ through the bijection $\iota \colon \mathring\Delta^{n-1} \to \Scal$, given by $(p_1, \dots, p_n) \mapsto \sum_{i=1}^{n} p_i \1_{\{i\}}$. Both these manifolds can be parametrised by the set
\begin{equation} \label{eq:Xi-categorica}
	\Xi = \left\{ \xi = (\xi^1,\dots,\xi^{n-1}) \in \R^{n-1} \suchthat \xi^i > 0,\ \textstyle\sum_{i=1}^{n-1} \xi^i < 1 \right\},
\end{equation}
by taking $p_i = \xi^i$, for $1 \le i \le n-1$, and $p_n = 1 - \sum_{i=1}^{n-1} \xi^i$.

To compute the Fisher matrix, it is useful to write $p(x) = \sum_{i=1}^{n-1} \xi^i \1_{\{i\}}(x) + \big( 1 -  \sum_{i=1}^{n-1} \xi^i \big) \1_{\{n\}}(x)$, so that
\begin{equation*}
	\partial_i \ell(\xi)
	= \frac{\1_{\{i\}}(x) - \1_{\{n\}}(x)}{\sum_{k=1}^{n} p_k \1_{\{k\}}(x)},
\end{equation*}
with $p_k = p_k(\xi)$ as above. The elements of the Fisher matrix are, for $1 \le i,j \le n-1$,
\begin{align} \label{eq:fisher-matrix-categorical}
	g_{ij}(\xi) &= \E \left[ \left( \partial_i \ell(\xi) \right) \left( \partial_j \ell(\xi) \right) \right]
	= \E \left[ \textstyle\frac{\left( \1_{\{i\}}(X) - \1_{\{n\}}(X) \right)\left( \1_{\{j\}}(X) - \1_{\{n\}}(X) \right)}{\left(\sum_k p_k \1_{\{k\}}(X)\right)^2} \right] \nonumber\\
	&= \E\left[ \textstyle\frac{\1_{\{i\}}(X) \1_{\{j\}}(X)}{\left(\sum_k p_k \1_{\{k\}}(X)\right)^2} 
	- \frac{\1_{\{i\}}(X) \1_{\{n\}}(X)}{\left(\sum_k p_k \1_{\{k\}}(X)\right)^2} 
	- \frac{\1_{\{j\}}(X) \1_{\{n\}}(X)}{\left(\sum_k p_k \1_{\{k\}}(X)\right)^2} 
	+ \frac{\left(\1_{\{n\}}(X)\right)^2}{\left(\sum_k p_k \1_{\{k\}}(X)\right)^2} \right]\nonumber\\
	&= \frac{\delta_{ij}}{p_i} + \frac{1}{p_n}\nonumber\\
	&= \frac{\delta_{ij}}{\xi^i} + \frac{1}{1 - \sum_{k=1}^{n-1} \xi^k}, 
\end{align}

\noindent where we have used that $\E\left[ \frac{\1_{\{i\}}(X) \1_{\{j\}}(X)}{\left(\sum_k p_k \1_{\{k\}}(X)\right)^2} \right] = \frac{\delta_{ij}}{p_i}$.

To obtain the geodesics and the Fisher--Rao distance, it is convenient to consider the mapping $p_i \mapsto 2\sqrt{p_i}$ that takes the simplex $\Delta^{n-1} \subseteq \R^{n}$ (in correspondence with the statistical manifold) to the positive part of the radius-two Euclidean sphere, denoted $S_{2,+}^{n-1} \subseteq \R^n$. In fact, this bijection is an isometry~\cite{ay2017,kass1997}:

\begin{proposition}
	The diffeomorphism
	\begin{align} \label{eq:f-categorical}
		f \colon \quad \Scal \subseteq \Pcal({\Xcal}) &\to S_{2,+}^{n-1} \subseteq \R^n\\
		p = \sum_{i=1}^{n} p_i \1_{\{i\}} &\mapsto (2\sqrt{p_1}, \dots, 2\sqrt{p_n}) \nonumber
	\end{align}
	is an isometry between the statistical manifold $\Scal$ equipped with the Fisher metric~$g_p$ and $S_{2,+}^{n-1}$ with the restriction of the ambient Euclidean metric.
\end{proposition}

\begin{proof}
	We will show that $g_p(u,v) = \left\langle \d f_p(u), \d f_p(v) \right\rangle$, for all $p \in \Scal$, $u,v \in T_p\Scal$. Let~$\varphi$ denote the parametrisation of the statistical manifold $\Scal$. Consider the curve $\alpha_i(t) = \varphi\left(\xi^1, \dots, \xi^i+t, \dots, \xi^{n-1} \right) \in \Scal$ and take $p=\varphi(\xi^1,\dots,\xi^{n-1})$. We have that
	\begin{equation*}
		\beta_i(t) \coloneqq \left( f \circ \alpha_i \right)(t) = \left( 2\sqrt{\xi^1}, \dots, 2\sqrt{\xi^i+t}, \dots, 2\sqrt{\xi^{n-1}}, 2\sqrt{1-\textstyle\sum_{k=1}^{n-1} \xi^k - t} \right).
	\end{equation*}
	
	\noindent Now, we can compute the differential applied to the tangent vector $\frac{\partial}{\partial\xi^i}(p)$:
	\begin{equation*}
		\d f_p \left( \frac{\partial}{\partial \xi^i}(p) \right)
		= \left. \frac{\d}{\d t} \beta_i (t) \right|_{t=0}
		= \left( 0, \dots, 0, \frac{1}{\sqrt{\xi^i}}, 0, \dots, 0, \frac{-1}{\sqrt{1-\textstyle\sum_{k=1}^{n-1}\xi^k}} \right).
	\end{equation*}
	
	\noindent Therefore,
	\begin{equation*}
		\left\langle \d f_p\left(\frac{\partial}{\partial\xi^i}(p)\right), \d f_p\left(\frac{\partial}{\partial\xi^j}(p)\right) \right\rangle
		= \frac{\delta_{ij}}{\xi^i} + \frac{1}{1-\sum_{k=1}^{n-1} \xi^k},
	\end{equation*}
	
	\noindent which is equal to $g_{ij}(\xi) = g_p \left( \frac{\partial}{\partial \xi^i}(p), \frac{\partial}{\partial \xi^j}(p) \right)$ in~\eqref{eq:fisher-matrix-categorical}. Since $\left\{ \frac{\partial}{\partial \xi^1}(p), \dots, \frac{\partial}{\partial \xi^{n-1}}(p) \right\}$ is a basis of $T_p\Scal$, this is enough to show that $f$ is indeed an isometry.
\end{proof}

Thus, the Fisher metric in $\Scal$ coincides with the Euclidean metric restricted to the positive part of the sphere~$S^{n-1}_{2,+}$, that is, the Fisher--Rao distance between distributions $p = \varphi(p_1, \dots, p_{n-1})$ and $q = \varphi(q_1, \dots, q_{n-1})$ in $\Scal$ is equal to the length of geodesic joining $f(p)$ and $f(q)$ on the sphere, which is great circle arc. This length is double the angle~$\alpha$ between the vectors $f(p)$ and $f(q)$, i.e., $2 \alpha = 2 \arccos \left\langle \frac{f(p)}{2}, \frac{f(q)}{2} \right\rangle = 2 \arccos \left(\sum_{i=1}^{n} \sqrt{p_iq_i}\right)$, with $p_n = 1-\sum_{i=1}^{n-1}p_i$ and $q_n = 1-\sum_{i=1}^{n-1}q_i$. Therefore, the Fisher--Rao distance between these two distributions is
\begin{equation} \label{eq:fr-distance-categorical}
	d_{\FR}(p,q) = 2 \arccos \left(\sum_{i=1}^{n} \sqrt{p_iq_i}\right).
\end{equation}

\noindent Note that the isometry~$f$ also allows extending the Fisher metric to the boundaries of the statistical manifold~$\Scal$.

\begin{remark}
	The reparametrisation to the sphere provides a nice geometrical interpretation for the relation between the Fisher--Rao distance and the Hellinger distance~\cite[\S~2.4]{tsybakov2009} $d_{\Hell}(p,q) = \sqrt{\sum_{i=1}^{n} \left( \sqrt{p_i} - \sqrt{q_i} \right)^2}$.	While the Fisher--Rao distance between distributions $p$ and $q$ is the length of the radius-two circumference arc between $f(p)$ and $f(q)$, the Hellinger distance is half the Euclidean distance between $f(p)$ and $f(q)$, i.e., $2d_{\Hell}(p,q) = \|f(p)-f(q)\|_2$. In other words, double the Hellinger distance is the arc-chord approximation for the Fisher--Rao distance.
\end{remark}

\subsection{Multinomial} \label{subsec:multinomial}

Consider $m$ i.i.d. experiments that follow a categorical distribution with $n$ possible outcomes and probabilities $p_1, \dots, p_n$. A multinomial distribution~\cite{atkinson1981,burbea1986,calin2014} gives the probability of getting $x_i$ times the $i$-th outcome, for $1 \le i \le n$ and $\sum_{i=1}^{n} x_i = m$. Its p.m.f. is $p(\pmb{x}) = p(x_1, \dots, x_n) = m! \prod_{i=1}^{n} \frac{p_i^{x_i}}{x_i!}$ and is defined on the sample space $\Xcal = \left\{ \pmb{x} = (x_1, \dots, x_n) \in \N^n \suchthat \sum_{i=1}^{n} x_i = m \right\}$. This distribution is parametrised by the same $\xi = (\xi^1, \dots, \xi^{n-1}) \in \Xi$ as in the categorical distribution, cf.~\eqref{eq:Xi-categorica}, with $p_i = \xi^i$, for $1 \le i \le n-1$ and $p_n = 1 - \sum_{i=1}^{n-1} \xi^i$. In this case, we have $\partial_i \ell(\xi) = \frac{x_i}{p_i} - \frac{x_n}{p_n}$ and $\partial_j\partial_i \ell(\xi) = - \frac{x_i}{p_i^2}\delta_{ij} - \frac{x_n}{p_n^2}$. Thus, the elements of the Fisher matrix are given by
\begin{align} \label{eq:fisher-matrix-multinomial}
	g_{ij}(\xi)
	&= - \E \left[ \partial_j\partial_i \ell(\xi) \right]
	= \E \left[ \frac{X_i}{p_i^2}\delta_{ij} + \frac{X_n}{p_n^2} \right]\nonumber\\
	&= m \left( \frac{\delta_{ij}}{p_i} + \frac{1}{p_n} \right)\nonumber\\
	&= m \left( \frac{\delta_{ij}}{\xi^i} + \frac{1}{1-\sum_{k=1}^{n-1} \xi^k} \right),
\end{align}

\noindent where we have used that $\E[X_i] = m p_i$. Note that this is the same metric as for the categorical distribution~\eqref{eq:fisher-matrix-categorical}, up to the factor $m$. Therefore, the Fisher--Rao distance between two distributions $p = \varphi(p_1, \dots, p_{n-1})$ and $q = \varphi(q_1, \dots, q_{n-1})$, with $p_n = 1-\sum_{i=1}^{n-1}p_i$ and $q_n = 1-\sum_{i=1}^{n-1}q_i$ is
\begin{equation}
	d_{\FR}(p,q) = 2 \sqrt{m} \arccos \left( \sum_{i=1}^{n} \sqrt{p_iq_i} \right).
\end{equation}

\subsection{Negative multinomial} \label{subsec:negative-multinomial}

A negative multinomial distribution~\cite{burbea1986,oller1985,khan2022} generalises the negative binomial distribution. Consider a sequence of i.i.d. categorical experiments with $n$ possible outcomes. A negative multinomial distribution models the number of times $x_1, \dots, x_{n-1}$ that the first $n-1$ outcomes occur before the $n$-th outcome occurs $x_n$ times. It is characterised by the p.m.f. $p(\pmb{x}) = p(x_1, \dots, x_{n-1}) = p_n^{x_n} \frac{\Gamma(\sum_{i=1}^{n} x_i)}{\Gamma(x_n)}  \prod_{i=1}^{n-1} \frac{p_i^{x_i}}{x_i!}$, and defined for $\pmb{x} = (x_1, \dots, x_{n-1}) \in \N^{n-1}$. It has the same parametrisation as the categorical distribution, cf.~\eqref{eq:Xi-categorica}, with $p_i = \xi^i$, for $1 \le i \le n-1$ and $p_n = 1 - \sum_{i=1}^{n-1} \xi^i$, for a fixed $x_n$ that can be extended to $x_n \in \R_+^*$. In this case, we have $\partial_i\ell(\xi) = \frac{x_i}{p_i} - \frac{x_n}{p_n}$ and $\partial_j\partial_i\ell(\xi) = -\frac{x_i}{p_i^2}\delta_{ij} -\frac{x_n}{p_n^2}$. The elements of the Fisher matrix are then
\begin{align}
	g_{ij}(\xi) &= -\E\left[ \partial_j\partial_i\ell(\xi) \right] = \E \left[  \frac{X_i}{p_i^2}\delta_{ij}  + \frac{x_n}{p_n^2} \right]\nonumber\\
	&= \frac{x_n}{p_n} \left( \frac{\delta_{ij}}{p_i} + \frac{1}{p_n} \right)\nonumber\\
	&= \frac{x_n}{1-\sum_{k=1}^{n-1}\xi^k} \left( \frac{\delta_{ij}}{\xi^i} + \frac{1}{1-\sum_{k=1}^{n-1}\xi^k} \right),
\end{align}
where we used that $\E[X_i] = x_np_i/p_n$.

To find the Fisher--Rao distance, we relate the geometry of this manifold to the radius-two hemisphere model with hyperbolic metric (cf. Section~\ref{sec:hyperbolic-geometry}) using a similar diffeomorphism as~\eqref{eq:f-categorical}. Consider
\begin{align} \label{eq:f-negative-multinomial}
	f  \colon \quad \Scal \subseteq \Pcal(\Xcal) &\to S_{2,+}^{n-1} \subseteq \Hcal^{n}\\
	p = \sum_{i=1}^{n} p_i \1_{\{i\}} &\mapsto \left( 2\sqrt{p_1}, \dots, 2\sqrt{p_n} \right). \nonumber
\end{align}

\noindent Denote $\varphi$ the parametrisation of the statistical manifold, and consider the point $p = \varphi(\xi^1,\dots,\xi^{n-1})$. Taking the curves $\alpha_i(t) = \varphi(\xi^1, \dots, \xi^i+t, \dots, \xi^{n-1})$ in $\Scal$ their images in $S_{2,+}^{n-1}$ are
\begin{equation*}
	\beta_i(t)
	\coloneqq (f \circ \alpha_i)(t)
	= \left( 2\sqrt{\xi^1}, \dots, 2\sqrt{\xi^i+t}, \dots, 2\sqrt{\xi^{n-1}}, 2\sqrt{1 - \textstyle\sum_{k=1}^{n-1}\xi^k -t} \right).
\end{equation*}
We have
\begin{equation*}
	\d f_p\left( \frac{\partial}{\partial\xi^i}(p) \right)
	= \left.\frac{\d}{\d t}\beta_i(t)\right|_{t=0}
	= \left(0, \dots, 0, \frac{1}{\sqrt{\xi^i}}, 0, \dots, \frac{-1}{\sqrt{1-\textstyle\sum_{k=1}^{n-1}\xi^k}} \right),
\end{equation*}
therefore
\begin{equation*}
	\left\langle \d f_p\left( \frac{\partial}{\partial\xi^i}(p) \right), \d f_p\left( \frac{\partial}{\partial\xi^j}(p) \right) \right\rangle_{G_{\Hcal^{n}}(p)}
	= \frac{1}{4\left(1-\sum_{k=1}^{n-1}\xi^k\right)} \left( \frac{\delta_{ij}}{\xi^i} + \frac{1}{1-\sum_{k=1}^{n-1}\xi^k} \right),
\end{equation*}
and we conclude that $g_{ij}(\xi) = g_{\xi} \left(  \frac{\partial}{\partial\xi^i}(p),  \frac{\partial}{\partial\xi^j}(p) \right) = 4x_n \left\langle \d f_p\left( \frac{\partial}{\partial\xi^i}(p) \right), \d f_p\left( \frac{\partial}{\partial\xi^j}(p) \right) \right\rangle_{G_{\Hcal^n}(p)}$. This means that, up to the factor $4x_n$, $f$ in~\eqref{eq:f-negative-multinomial} is an isometry. Then, using~\eqref{eq:distance-hyperbolic-hemisphere-r}, we find that the Fisher--Rao distance between two distributions $p = \varphi(p_1, \dots, p_{n-1})$ and $q = \varphi(q_1, \dots, q_{n-1})$, with $p_n = 1-\sum_{i=1}^{n-1}p_i$ and $q_n = 1-\sum_{i=1}^{n-1}q_i$, is
\begin{align}
	d_\FR(p,q)
	% &= 2 \sqrt{x_n} d_{\Hcal^{n}}\left( (p_1,\dots,p_n), (q_1,\dots,q_n) \right) \nonumber\\
	&= 2 \sqrt{x_n}\ \tilde{d}_{S_2^{n-1}}\left( (2\sqrt{p_1},\dots,2\sqrt{p_n}), (2\sqrt{q_1},\dots,2\sqrt{q_n}) \right) \nonumber\\
	&= 2 \sqrt{x_n} \arccosh \left( \frac{1 - \sum_{i=1}^{n-1} \sqrt{p_iq_i}}{\sqrt{p_nq_n}} \right).
\end{align}

\noindent In the above equality, we used the fact that the hyperbolic distance in $S_{2,+}^{n-1} \subseteq \Hcal^n$ is invariant to dilation and contraction, as remarked in Section~\ref{sec:hyperbolic-geometry}. We also conclude that the geodesics in the statistical manifold~$\Scal$ are associated to orthogonal semicircles in $S_{2,+}^{n-1}$.

It is interesting to note that the similar maps \eqref{eq:f-categorical} and \eqref{eq:f-negative-multinomial}, respectively, for categorical and negative multinomial distributions are used to embed the statistical manifold (in correspondence with the simplex) in the radius-two sphere, but with different metrics. Figure~\ref{fig:geodesics-discrete} illustrates the geodesics according to these two metrics.

\begin{figure}
	\centering
	\begin{subfigure}[b]{0.3\textwidth}
		\centering
		\includegraphics[width=\linewidth]{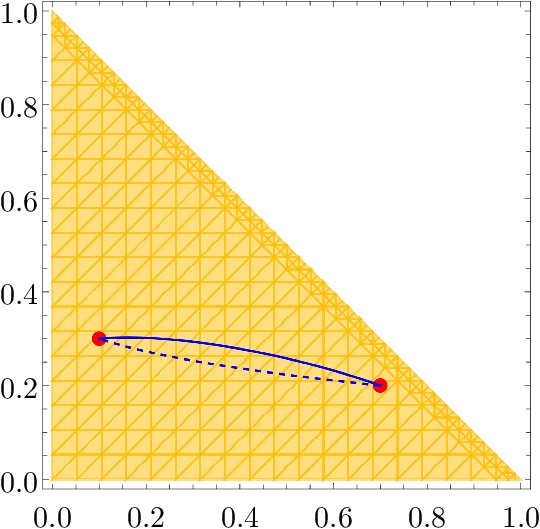}
		\caption{Parameter space $\Xi \subseteq \R^2$.}
	\end{subfigure}
	\hfill
	\begin{subfigure}[b]{0.3\textwidth}
		\centering
		\includegraphics[height=\linewidth]{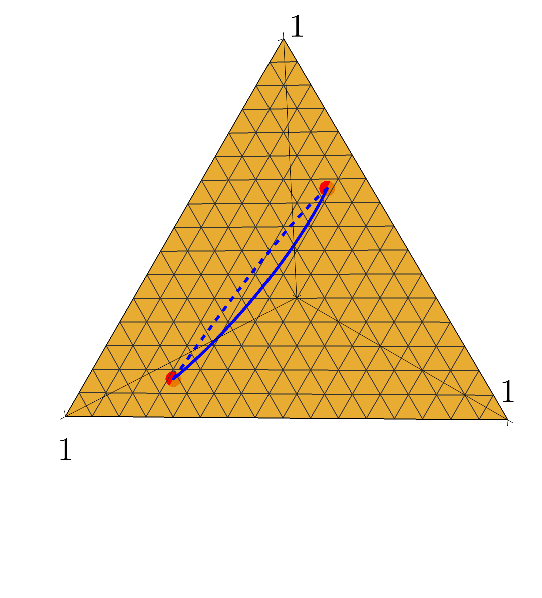}
		\caption{Simplex~$\Delta^2$.}
	\end{subfigure}
	\hfill
	\begin{subfigure}[b]{0.3\textwidth}
		\centering
		\includegraphics[height=\linewidth]{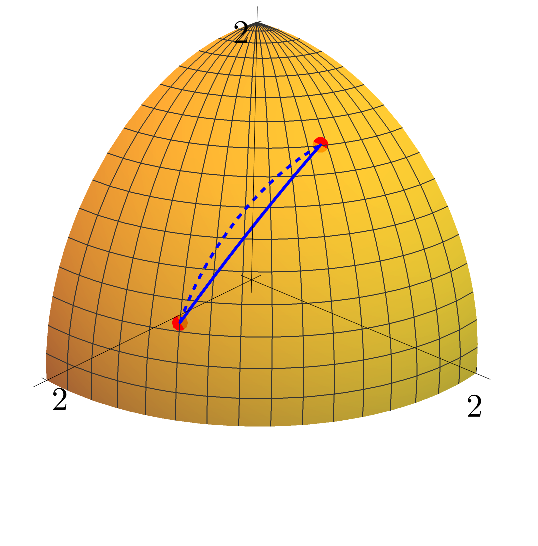}
		\caption{Sphere~$S^2_{2,+}$.}
	\end{subfigure}
	\caption{
		Geodesics joining points $p=(0.7,0.2,0.1)$ and $q=(0.1,0.3,0.6)$ according to categorical (solid) and negative multinomial (dashed) metrics, seen in the parameter space, on the simplex, and on the sphere. The distance between the categorical distributions is $d_{\text{cat}}(p,q) \approx 1.432$, and between the negative multinomial distributions is $d_{\text{neg-mult}}(p,q) \approx 2.637 \sqrt{x_n}$.
	}
	\label{fig:geodesics-discrete}
\end{figure}

\section{Continuous Distributions} \label{sec:continuous-distributions}

In the following, we start presenting one-dimensional examples (\S~\ref{subsec:exponential}---\ref{subsec:erlang}), then we consider two-dimensional statistical manifolds (\S~\ref{subsec:elliptical}---\ref{subsec:power-function}), and multivariate models (\S~\ref{subsec:wishart}--\ref{subsec:inverse-wishart}). The results are summarised in Tables~\ref{tab:continuous-distributions} and~\ref{tab:continuous-distributions-2}.

\begin{sidewaystable}
	\centering
	\caption{Fisher--Rao distances for continuous distributions.}
	\label{tab:continuous-distributions}
	\def\arraystretch{1.5}\tabcolsep=3pt
	\footnotesize
	\begin{tabular}{ccccccc}
		\hline
		& Distribution & Support & Parameters & Fisher matrix & Fisher--Rao distance & References\\ \hline\hline \vspace{0.5em}
		
		Exponential
		& $\lambda e^{-\lambda x}$
		& $x \in \R_+$
		& $\lambda \in \R_+^*$
		& $\begin{pmatrix} \frac{1}{\lambda^2} \end{pmatrix}$
		& $\left| \log \lambda_1 - \log \lambda_2 \right|$
		& \cite{atkinson1981} \\ \vspace{0.5em}
		
		Rayleigh
		& $\frac{x}{\sigma^2} \exp \left( -\frac{x^2}{2\sigma^2} \right)$
		& $x \in \R_+$
		& $\sigma \in \R_+^*$
		& $\begin{pmatrix} \frac{4}{\sigma^2} \end{pmatrix}$
		& $2 \left| \log \sigma_1 - \log \sigma_2 \right|$
		& $*$ \\ \vspace{0.5em}
		
		Erlang
		& $\lambda^k x^{k-1} e^{-\lambda x}/(k-1)!$
		& $x \in \R_+$
		& $\lambda \in \R_+^*$
		& $\begin{pmatrix} \frac{k}{\lambda^2} \end{pmatrix}$
		& $\sqrt{k} \left| \log \lambda_1 - \log \lambda_2 \right|$
		& $*$ \\ \vspace{0.5em}
		
		Gaussian & $\frac{1}{\sigma\sqrt{2\pi}}\exp\left(-\frac{(x-\mu)^2}{2\sigma^2}\right)$
		& $x \in \R$
		& $(\mu,\sigma) \in \R \times \R_+^*$
		& $\begin{pmatrix} \frac{1}{\sigma^2} & 0\\ 0 & \frac{2}{\sigma^2}\end{pmatrix}$
		& $2\sqrt{2} \arctanh \left( \sqrt{\frac{(\mu_1-\mu_2)^2 + 2(\sigma_1-\sigma_2)^2}{(\mu_1-\mu_2)^2 + 2(\sigma_1+\sigma_2)^2}} \right)$
		&\shortstack{\cite{calin2014,atkinson1981,burbea1986}\\\cite{mithcell1988,arwini2008,costa2015}}\\ \vspace{0.5em}
		
		Laplace
		& {$\frac{1}{2\sigma}\exp\left( -\frac{|x-\mu|}{\sigma} \right)$}
		& $x \in \R$
		& $(\mu,\sigma) \in \R \times \R_+^*$
		& $\begin{pmatrix}
			\frac{1}{\sigma^2} & 0 \\
			0 & \frac{1}{\sigma^2}
		\end{pmatrix}$
		& $2\arctanh \left( \sqrt{\frac{(\mu_1-\mu_2)^2+(\sigma_1-\sigma_2)^2}{(\mu_1-\mu_2)^2+(\sigma_1+\sigma_2)^2}} \right)$
		& \shortstack{Particular\\ case in~\cite{verdoolaege2012}}\\ \vspace{0.5em}
		
		\shortstack{Generalised\\Gaussian} 
		& $\frac{\beta}{2\sigma \Gamma(1/\beta)}\exp\left(-\frac{\vert x-\mu\vert^{\beta}}{\sigma}\right)$
		& $x \in \R$
		& $(\mu, \sigma) \in \R \times \R_+^*$
		& $\begin{pmatrix}
			\frac{\beta}{\sigma^2}\frac{\Gamma(2-\frac{1}{\beta})}{\Gamma(1+\frac{1}{\beta})} & 0\\
			0 & \frac{\beta}{\sigma^2}
		\end{pmatrix}$
		& $\sqrt{\beta+1} \arctanh \left(
		\sqrt{\frac{
				(\mu_1-\mu_2)^2\Gamma\left(2-\frac{1}{\beta}\right) + \frac{\beta+1}{\beta}(\sigma_1-\sigma_2)^2\Gamma\left(1+\frac{1}{\beta}\right)
			}{
				(\mu_1-\mu_2)^2\Gamma\left(2-\frac{1}{\beta}\right) + \frac{\beta+1}{\beta}(\sigma_1+\sigma_2)^2\Gamma\left(1+\frac{1}{\beta}\right)
		}}
		\right)$
		& \shortstack{Particular\\case in~\cite{verdoolaege2012}} \\ \vspace{0.5em}
		
		Logistic
		& $\frac{\exp\left(-(x-\mu)/\sigma\right)}{\sigma\left(\exp\left(-(x-\mu)/\sigma\right)+1\right) ^2}$
		& $x \in \R$
		& $(\mu, \sigma) \in \R \times \R^*_+$
		& $\begin{pmatrix}
			\frac{1}{3\sigma^2} & 0\\
			0 & \frac{\pi^2+3}{9\sigma^2}
		\end{pmatrix}$
		& $\frac{2\sqrt{\pi^2+3}}{3} \arctanh \left( \sqrt{\frac{3(\mu_2-\mu_1)^2+(\pi^2+3)(\sigma_2-\sigma_1)^2}{3(\mu_2-\mu_1)^2+(\pi^2+3)(\sigma_2+\sigma_1)^2}} \right)$
		& \cite{oller1987} \\ \vspace{0.5em}
		
		Cauchy
		& $\dfrac{\sigma}{\pi \left[ (x-\mu)^2 + \sigma^2 \right]}$
		& $x \in \R$
		& $(\mu, \sigma) \in \R \times \R_+^*$
		& $\begin{pmatrix}
			\frac{1}{2\sigma^2} & 0\\
			0 & \frac{1}{2\sigma^2}
		\end{pmatrix}$
		& ${\sqrt{2}} \arctanh \left( \sqrt{\frac{(\mu_1-\mu_2)^2+(\sigma_1-\sigma_2)^2}{(\mu_1-\mu_2)^2+(\sigma_1+\sigma_2)^2}} \right)$
		& \cite{mithcell1988,nielsen2020-cauchy} \\ \vspace{0.5em}
		
		\shortstack{Student's~$t$\\\phantom{a}}
		& \shortstack{
			$\left(1+ \frac{1}{\nu}(\frac{x-\mu}{\sigma})^2 \right)^{-\frac{\nu+1}{2}}$\\
			$\times\frac{\Gamma\left((\nu+1)/2\right)}{\sigma\sqrt{\pi\nu}\Gamma(\nu/2)}$
		}
		%		& {$\left(1+ \frac{1}{\nu}(\frac{x-\mu}{\sigma})^2 \right)^{-\frac{\nu+1}{2}}  \frac{\Gamma\left((\nu+1)/2\right)}{\sigma\sqrt{\pi\nu}\Gamma(\nu/2)}$}
		& $x \in \R$
		& $(\mu,\sigma) \in \R \times \R_+^*$
		& $\begin{pmatrix} 
			\frac{\nu+1}{(\nu+3)\sigma^2} & 0\\
			0 & \frac{2\nu}{(\nu+3)\sigma^2}
		\end{pmatrix}$
		& $2\sqrt{\frac{2\nu}{\nu+3}} \arctanh \left( \sqrt{\frac{(\nu+1)(\mu_2-\mu_1)^2+2\nu(\sigma_2-\sigma_1)^2}{(\nu+1)(\mu_2-\mu_1)^2+2\nu(\sigma_2+\sigma_1)^2}} \right)$
		& \cite{mithcell1988} \\ \hline
	\end{tabular}
\end{sidewaystable}

\begin{sidewaystable}
	\centering
	\caption{Fisher--Rao distances for continuous distributions (continued).}
	\label{tab:continuous-distributions-2}
	\def\arraystretch{1.5}\tabcolsep=3pt
	\footnotesize
	\begin{tabular}{ccccccc}
		\hline
		& Distribution & Support & Parameters & Fisher matrix & Fisher--Rao distance & References\\ \hline\hline \vspace{0.5em}
		
		\shortstack{Log-\\Gaussian}
		& $\frac{1}{\sigma x \sqrt{2\pi}}\exp\left(-\frac{(\log x-\mu)^2}{2\sigma^2}\right)$
		& $x \in \R_+^*$
		& $(\mu,\sigma) \in \R \times \R_+^*$
		& $\begin{pmatrix} \frac{1}{\sigma^2} & 0\\ 0 & \frac{2}{\sigma^2}\end{pmatrix}$
		& $2\sqrt{2} \arctanh \left( \sqrt{\frac{(\mu_1-\mu_2)^2 + 2(\sigma_1-\sigma_2)^2}{(\mu_1-\mu_2)^2 + 2(\sigma_1+\sigma_2)^2}} \right)$
		& \cite{calin2014,arwini2008} \\ \vspace{0.5em}
		
		\shortstack{Inverse\\Gaussian}
		& $\sqrt{\frac{\lambda}{2\pi x^3}} \exp \left( - \frac{\lambda(x - \mu)^2}{2\mu^2 x} \right)$
		& $x \in \R_+^*$
		& $(\lambda,\mu) \in \R_+^* \times \R_+^*$
		& $\begin{pmatrix}
			\frac{1}{2\lambda^2} & 0\\
			0 & \frac{\lambda}{\mu^3}
		\end{pmatrix}$
		& $2\sqrt{2} \arctanh \left(
		\sqrt{\frac{
				\mu_1\mu_2(\sqrt{\lambda_1}-\sqrt{\lambda_2})^2+2\lambda_1\lambda_2(\sqrt{\mu_1}-\sqrt{\mu_2})^2
			}{
				\mu_1\mu_2(\sqrt{\lambda_1}-\sqrt{\lambda_2})^2+2\lambda_1\lambda_2(\sqrt{\mu_1}+\sqrt{\mu_2})^2}
		} \right)$
		& \cite{villarroya1993,khan2022}  \\ \vspace{0.5em}
		
		Gumbel
		& $\frac{1}{\sigma} e^{-\frac{x-\mu}{\sigma}} \exp\left(-e^{-\frac{x-\mu}{\sigma}}\right)$
		& $x \in \R$
		& $(\mu, \sigma) \in \R \times \R^*_+$
		& {$\begin{pmatrix}
				\frac{1}{\sigma^2} & \frac{\gamma-1}{\sigma^2}\\
				\frac{\gamma-1}{\sigma^2} & \frac{(\gamma-1)^2+\frac{\pi^2}{6}}{\sigma^2}
			\end{pmatrix}$}
		& $\frac{2\pi}{\sqrt{6}} \arctanh \left( \sqrt{
			\frac{
				\left[(\mu_1-\mu_2)-(1-\gamma)(\sigma_1-\sigma_2)\right]^2+\frac{\pi^2}{6}(\sigma_1-\sigma_2)^2
			}{
				\left[(\mu_1-\mu_2)-(1-\gamma)(\sigma_1-\sigma_2)\right]^2+\frac{\pi^2}{6}(\sigma_1+\sigma_2)^2
			}
		} \right)$
		& \cite{oller1987} \\ \vspace{0.5em}
		
		Fréchet
		& $\frac{\lambda}{\beta} \left( \frac{x}{\beta} \right)^{-\lambda-1} e^{- \left(\frac{x}{\beta}\right)^{-\lambda}}$
		& $x \in \R_+^*$
		& $(\beta,\lambda) \in \R_+^* \times \R_+^*$
		& {$\begin{pmatrix}
				\frac{\lambda^2}{\beta^2} & \frac{1-\gamma}{\beta}\\
				\frac{1-\gamma}{\beta} & \frac{(\gamma-1)^2+\frac{\pi^2}{6}}{\lambda^2}
			\end{pmatrix}$}
		& {$\frac{2\pi}{\sqrt{6}} \arctanh \left( \sqrt{
				\frac{
					\left[\log\frac{\beta_1}{\beta_2}-(1-\gamma)\left(\frac{1}{\lambda_1}-\frac{1}{\lambda_2}\right)\right]^2+\frac{\pi^2}{6}\left(\frac{1}{\lambda_1}-\frac{1}{\lambda_2}\right)^2
				}{
					\left[\log\frac{\beta_1}{\beta_2}-(1-\gamma)\left(\frac{1}{\lambda_1}-\frac{1}{\lambda_2}\right)\right]^2+\frac{\pi^2}{6}\left(\frac{1}{\lambda_1}+\frac{1}{\lambda_2}\right)^2
				}
			} \right)$}
		& \cite{oller1987} \\ \vspace{0.5em}
		
		Weibull
		& $\frac{\lambda}{\beta} \left( \frac{x}{\beta} \right)^{\lambda-1} e^{- \left(\frac{x}{\beta}\right)^{\lambda}}$
		& $x \in \R_+^*$
		& $(\beta,\lambda) \in \R_+^* \times \R_+^*$
		& {$\begin{pmatrix}
				\frac{\lambda^2}{\beta^2} & \frac{\gamma-1}{\beta}\\
				\frac{\gamma-1}{\beta} & \frac{(\gamma-1)^2+\frac{\pi^2}{6}}{\lambda^2}
			\end{pmatrix}$}
		& {$\frac{2\pi}{\sqrt{6}} \arctanh \left( \sqrt{
				\frac{
					\left[\log\frac{\beta_2}{\beta_1}-(1-\gamma)\left(\frac{1}{\lambda_1}-\frac{1}{\lambda_2}\right)\right]^2+\frac{\pi^2}{6}\left(\frac{1}{\lambda_1}-\frac{1}{\lambda_2}\right)^2
				}{
					\left[\log\frac{\beta_2}{\beta_1}-(1-\gamma)\left(\frac{1}{\lambda_1}-\frac{1}{\lambda_2}\right)\right]^2+\frac{\pi^2}{6}\left(\frac{1}{\lambda_1}+\frac{1}{\lambda_2}\right)^2
				}
			} \right)$}
		& \shortstack{\cite{burbea1986,oller1987,wauters1993}} \\ \vspace{0.5em}
		
		Pareto
		& ${\theta}\alpha^\theta{x^{-(\theta+1)}}$
		& $x \in \left[\alpha,\infty\right[$
		& $(\theta, \alpha) \in \R^*_+ \times \R^*_+$
		& $\begin{pmatrix}
			\frac{1}{\theta^2} & 0\\
			0 & \frac{\theta^2}{\alpha^2}
		\end{pmatrix}$
		& $2 \arctanh \left( \sqrt{\frac{(\theta_1\theta_2\log(\alpha_1/\alpha_2))^2+(\theta_1-\theta_2)^2}{(\theta_1\theta_2\log(\alpha_1/\alpha_2))^2+(\theta_1+\theta_2)^2}} \right)$
		& \cite{burbea1986,li2022} \\ \vspace{0.5em}
		
		\shortstack{Power\\function}
		& $\theta \beta^{-\theta} x^{\theta-1}$
		& $x \in \left]0,\beta\right]$
		& $(\theta,\beta) \in \R_+^* \times \R_+^*$
		& $\begin{pmatrix}
			\frac{1}{\theta^2} & 0\\
			0 & \frac{\theta^2}{\beta^2}
		\end{pmatrix}$
		& $2 \arctanh \left( \sqrt{\frac{(\theta_1\theta_2\log(\beta_{1}/\beta_{2}))^2+(\theta_1-\theta_2)^2}{(\theta_1\theta_2\log(\beta_{1}/\beta_{2}))^2+(\theta_1+\theta_2)^2}} \right)$
		& \shortstack{Particular\\case in~\cite{burbea1986}} \\ \vspace{0.5em}
		
		Wishart
		& see below
		& $X\in P_m(\R)$
		& \shortstack{same as for\\$\Sigma \in P_m(\R)$}
		& $\frac{n}{2} D_m^\t \left(\Sigma^{-1}\otimes\Sigma^{-1}\right)D_m$
		& $\sqrt{\frac{n}{2}\sum_{k=1}^m \left(\log \lambda_k \right)^2}$
		& \cite{ayadi2023} \\ \vspace{0.5em}
		
		\shortstack{Inverse\\Wishart}
		& see below
		& $X\in P_m(\R)$
		& \shortstack{same as for\\$\Sigma \in P_m(\R)$}
		& $\frac{n}{2} D_m^\t \left(\Sigma^{-1}\otimes\Sigma^{-1}\right)D_m$
		& $\sqrt{\frac{n}{2}\sum_{k=1}^m \left(\log \lambda_k \right)^2}$
		& * \\ \hline
	\end{tabular}
	{
		\begin{flushleft}
		Wishart distribution: $\frac{{(\det X )}^{(n-m-1)/2}\exp\left(-\frac{1}{2}\tr(\Sigma^{-1}X)\right)}{2^{nm/2}(\det \Sigma)^{n/2}\Gamma_m(n/2)}$\\
		Inverse Wishart distribution: $\frac{{(\det \Sigma)^{n/2} (\det X )}^{-(n+m+1)/2}\exp\left(-\frac{1}{2}\tr(\Sigma X^{-1})\right)}{2^{nm/2}\Gamma_m(n/2)}$
		\end{flushleft}
	}
\end{sidewaystable}

\subsection{Exponential} \label{subsec:exponential}

An exponential distribution~\cite{atkinson1981} has p.d.f. $p(x) = \lambda e^{-\lambda x}$, defined for $x \in \R_+$, and parametrised by $\lambda \in \R_+^{*}$. In this case, we have $\partial_{\lambda}\ell(\lambda) = \frac{1}{\lambda} - x$, and the Fisher information is
\begin{align}
	g_{11}(\lambda)
	&= \E \left[ \left(\partial_{\lambda} \ell(\lambda) \right)^2 \right]
	= \E \left[ \left( \frac{1}{\lambda} - X \right)^2 \right] \nonumber\\
	&= \frac{1}{\lambda^2} - \frac{2}{\lambda}\E[X] + \E[X^2] \nonumber\\
	&= \frac{1}{\lambda^2},
\end{align}

\noindent where we have used that $\E[X]=\frac{1}{\lambda}$ and $\E[X^2] = \frac{2}{\lambda^2}$. The Fisher--Rao distance is given by
\begin{equation}
	d_{\FR}(\lambda_1,\lambda_2)
	= \left| \int_{\lambda_1}^{\lambda_2} \frac{1}{\lambda}~\d \lambda \right| 
	= \left| \log \lambda_1 - \log \lambda_2 \right|.
\end{equation}

\subsection{Rayleigh} \label{subsec:rayleigh}

A Rayleigh distribution has p.d.f. $p(x) = \frac{x}{\sigma^2}\exp\left( -\frac{x^2}{2\sigma^2} \right)$, defined for $x \in \R_+$, and parametrised by $\sigma \in \R_+^*$. We have $\partial_{\sigma} \ell(\sigma) = \frac{x^2}{\sigma^3} - \frac{2}{\sigma}$, and the Fisher information is
\begin{align}
	g_{11}(\sigma)
	&= \E \left[ \left(\partial_{\sigma} \ell(\sigma) \right)^2 \right]
	= \E \left[ \left( \frac{X^2}{\sigma^3} - \frac{2}{\sigma} \right)^2 \right]\nonumber\\
	&= \frac{\E[X^4]}{\sigma^6} - \frac{2\E[X^2]}{\sigma^4} + \frac{4}{\sigma^2}\nonumber\\
	&= \frac{4}{\sigma^2},
\end{align}

\noindent where we have used that $\E[X^2] = 4\sigma^2$ and $\E[X^4] = 8\sigma^4$. Thus, the Fisher--Rao distance is given by
\begin{equation}
	d_{\FR}(\sigma_1,\sigma_2)
	= \left| \int_{\sigma_1}^{\sigma_2} \frac{2}{\sigma}~\d \sigma \right| 
	= 2 \left| \log \sigma_1 - \log \sigma_2 \right|.
\end{equation}

\subsection{Erlang} \label{subsec:erlang}

An Erlang distribution has p.d.f. $p(x) = \frac{\lambda^k x^{k-1} e^{-\lambda x}}{(k-1)!}$, defined for $x \in \R_+$, and parametrised by $\lambda \in \R_+^*$, for a fixed $k \in \N^*$. We have $\partial_\lambda \ell(\lambda) = \frac{k}{\lambda} - x$, so that the Fisher information is
\begin{align}
	g_{11}(\lambda)
	&= \E \left[ \left(\partial_{\lambda} \ell(\lambda) \right)^2 \right]
	= \E \left[ \left( \frac{k}{\lambda} - X \right)^2 \right]\nonumber\\
	&= \E[X^2] - \frac{2k\E[X]}{\lambda} + \frac{k^2}{\lambda^2}\nonumber\\
	&= \frac{k}{\lambda^2},
\end{align}

\noindent where we have used that $\E[X] = \frac{k}{\lambda}$ and $\E[X^2] = \frac{k(k+1)}{\lambda^2}$. The Fisher--Rao distance is then
\begin{equation}
	d_{\FR}(\lambda_1,\lambda_2)
	= \left| \int_{\lambda_1}^{\lambda_2} \frac{\sqrt{k}}{\lambda}~\d \lambda \right| 
	= \sqrt{k} \left| \log \lambda_1 - \log \lambda_2 \right|.
\end{equation}

\subsection{Univariate elliptical distributions} \label{subsec:elliptical}

Elliptical distributions are a class of distributions that generalise Gaussian distributions~\cite{fang1990}. Here we focus on univariate elliptical distributions, which are defined for $x \in \R$, parametrised by $(\mu, \sigma) \in \R \times \R_+^{*}$, and have p.d.f. of the form
\begin{equation} \label{eq:elliptical-distribution}
	p(x) = \frac{1}{\sigma} h \left( \frac{(x-\mu)^2}{\sigma^2} \right),
\end{equation}
for a fixed measurable function $h \colon \R_+ \to \R_+$ that satisfies $\int_{-\infty}^{\infty} h(z^2)~\d z = 1$ and $\lim_{z\to+\infty} zh(z^2)=0$. For a random variable~$X$ distributed according to~\eqref{eq:elliptical-distribution}, provided its mean and variance exist, they are given by $\E[X]=\mu$ and $\text{Var}(X) = \sigma^2 \int_{-\infty}^{\infty} z^2h(z^2)~\d z$. For each function $h$, the set of distributions of the form~\eqref{eq:elliptical-distribution} forms a statistical manifold parametrised by $(\mu,\sigma)$. Some examples of these manifolds have been studied in~\cite{mithcell1988}, and we write some other examples in the same standard form.

For these distributions, we have $\partial_{\mu} \ell \coloneqq \partial_{\mu} \ell(\mu,\sigma) = -\frac{2(x-\mu)}{\sigma^2} \frac{h' \left(\sigma^{-2}(x-\mu)^2 \right)}{h \left(\sigma^{-2}(x-\mu)^2 \right)}$, and $\partial_{\sigma} \ell \coloneqq \partial_{\sigma} \ell(\mu,\sigma) = -\frac{1}{\sigma} - \frac{2(x-\mu)^2}{\sigma^3} \frac{h' \left(\sigma^{-2}(x-\mu)^2 \right)}{h \left(\sigma^{-2}(x-\mu)^2 \right)}$. Thus, the elements of the Fisher matrix are

\begin{align*}
	g_{11} &= \E \left[ \left( \partial_{\mu}\ell \right)^2 \right]\\
	%	\right]\\
	&= \int_{-\infty}^{\infty}  \left( -\frac{2(x-\mu)}{\sigma^2} \frac{h' \left(\sigma^{-2}(x-\mu)^2 \right)}{h \left(\sigma^{-2}(x-\mu)^2 \right)} \right)^2 \frac{1}{\sigma} h \left( \frac{(x-\mu)^2}{\sigma^2} \right)~\d x\\
	&=\frac{4}{\sigma^2}\int_{-\infty}^{\infty} z^2 \frac{\left[h'(z^2)\right]^2}{h(z^2)}~\d z,\\
	g_{12} = g_{21} &= \E \left[ \left(\partial_\mu \ell\right) \left(\partial_\sigma \ell\right) \right]\\
	&=\int^{\infty}_{-\infty}
	\left(
	-\frac{2(x-\mu)}{\sigma^2} \frac{h' \left(\sigma^{-2}(x-\mu)^2 \right)}{h \left(\sigma^{-2}(x-\mu)^2 \right)}
	\right)
	\left(
	-\frac{1}{\sigma} - \frac{2(x-\mu)^2}{\sigma^3} \frac{h' \left(\sigma^{-2}(x-\mu)^2 \right)}{h \left(\sigma^{-2}(x-\mu)^2 \right)} 
	\right)\\
	&\hspace{4em}\times \frac{1}{\sigma} h \left( \frac{(x-\mu)^2}{\sigma^2} \right)~\d x\\
	&=\frac{2}{\sigma^2} \int^{\infty}_{-\infty} z\left(h'(z^2) + 2z^2 \frac{\left[h'(z^2)\right]^2}{h(z^2)}\right)~\d z\\
	&=0,\\
	g_{22} &=\E \left[ \left( \partial_{\sigma}\ell \right)^2 \right]\\
	&= \int_{-\infty}^{\infty}  \left(
	-\frac{1}{\sigma} - \frac{2(x-\mu)^2}{\sigma^3} \frac{h' \left(\sigma^{-2}(x-\mu)^2 \right)}{h \left(\sigma^{-2}(x-\mu)^2 \right)}
	\right)^2 \frac{1}{\sigma} h \left( \frac{(x-\mu)^2}{\sigma^2} \right)~\d x\\
	&= \int_{-\infty}^{\infty} \left( \frac{1}{\sigma^2} + \frac{4}{\sigma^2}z^4 \left( \frac{h'(z^2)}{h(z^2)} \right)^2 + \frac{4}{\sigma^2} z^2 \frac{h'(z^2)}{h(z^2)} \right) h(z^2)~\d z\\
	&= \frac{1}{\sigma^2}
	+ \frac{4}{\sigma^2} \int_{-\infty}^{\infty} z^4 \frac{\left[h'(z^2)\right]^2}{h(z^2)}~\d z
	+ \frac{4}{\sigma^2} \int_{-\infty}^{\infty} z^2 h'(z^2)~\d z \\
	&=\frac{4}{\sigma^2} \int_{-\infty}^{\infty} z^4 \frac{\left[h'(z^2)\right]^2}{h(z^2)}~\d z-\frac{1}{\sigma^2}.
\end{align*}

The Fisher matrix for elliptical distributions has the form
\begin{equation} \label{eq:elliptical-metric}
	G(\mu,\sigma)
	= 
	\begin{pmatrix}
		\frac{a_h}{\sigma^2} & 0\\
		0 & \frac{b_h}{\sigma^2}
	\end{pmatrix},
\end{equation}
where
\begin{equation*}
	a_h \coloneqq 4 \int_{-\infty}^{\infty} z^2 \frac{\left[h'(z^2)\right]^2}{h(z^2)}~\d z,
\end{equation*}
and
\begin{equation*}
	b_h \coloneqq 4\int_{-\infty}^{\infty} z^4 \frac{\left[h'(z^2)\right]^2}{h(z^2)}~\d z - 1.
\end{equation*}

Applying Lemma~\ref{lemma:poincare-isometry} to a statistical manifold formed by univariate elliptical distributions (i.e., for a fixed~$h$) equipped with the metric~\eqref{eq:elliptical-metric}, we get an expression for the Fisher--Rao distance in this manifold given by
\begin{align}\label{eq:fr-distance-elliptical}
	d_{\FR} \big((\mu_1,\sigma_1), (\mu_2, \sigma_2)\big)
	&= \sqrt{b_h} d_{\Hcal^2} \left( (\sqrt{a_h}\mu_1,\sqrt{b_h}\sigma_1),( \sqrt{a_h}\mu_2,\sqrt{b_h}\sigma_2 )\right)\nonumber\\
	&= 2\sqrt{b_h} \arctanh \left( \sqrt{\frac{a_h(\mu_1-\mu_2)^2 + b_h(\sigma_1-\sigma_2)^2}{a_h(\mu_1-\mu_2)^2 + b_h(\sigma_1+\sigma_2)^2}} \right).
\end{align}

\noindent Note that, if $\mu_1 = \mu_2 \eqqcolon \mu$, the expression simplifies to
\begin{align}
	d_{\FR} \big((\mu,\sigma_1), (\mu, \sigma_2)\big)
	&= \sqrt{b_h} \left\vert \log \sigma_1 - \log \sigma_2 \right\vert.
\end{align}

\subsubsection{Gaussian} \label{subsec:gaussian}

A Gaussian distribution~\cite{atkinson1981,arwini2008,burbea1986,calin2014,costa2015,mithcell1988} is characterised by the p.d.f. $p(x) = \frac{1}{\sigma \sqrt{2 \pi}}\exp\left(-\frac{(x-\mu)^2}{2\sigma^2}\right)$, defined for $x \in \R$, and parametrised by the pair $(\mu, \sigma) \in \R \times \R_+^*$. Note that Gaussian distributions are elliptical distribution with $h(u)=\frac{1}{\sqrt{2\pi}}\exp(-u/2)$, $a_h=1$, and $b_h=2$. Thus, by~\eqref{eq:elliptical-metric}, the corresponding Fisher matrix is
\begin{equation} \label{eq:fisher-matrix-gaussian}
	G(\mu,\sigma) = \begin{pmatrix}
		\frac{1}{\sigma^2} & 0\\
		0 & \frac{2}{\sigma^2}
	\end{pmatrix}.
\end{equation}
And, by~\eqref{eq:fr-distance-elliptical}, the Fisher--Rao distance is obtained as 
\begin{align}
	d_{\FR} \big((\mu_1,\sigma_1), (\mu_2, \sigma_2)\big)
	&= 2\sqrt{2} \arctanh \left( \sqrt{\frac{(\mu_1-\mu_2)^2 + 2(\sigma_1-\sigma_2)^2}{(\mu_1-\mu_2)^2 + 2(\sigma_1+\sigma_2)^2}} \right).
\end{align}

\subsubsection{Laplace} \label{subsec:laplace}

A Laplace distribution has p.d.f. $p(x) = \frac{1}{2\sigma}\exp\left( -\frac{|x-\mu|}{\sigma} \right)$, defined for $x \in \R$, and parametrised by $(\mu,\sigma) \in \R \times \R_+^*$. Laplace distributions are elliptical distribution with $h(u)=\frac{1}{2}\exp\left(-\sqrt{u}\right)$, $a_h=1$, and $b_h=1$. Using~\eqref{eq:elliptical-metric}, we get the Fisher matrix as
\begin{equation}
	G(\mu,\sigma) = \begin{pmatrix}
		\frac{1}{\sigma^2} & 0\\
		0 & \frac{1}{\sigma^2}
	\end{pmatrix},
\end{equation}
showing that this metric coincides with the hyperbolic metric~\eqref{eq:hyperbolic-metric} of the Poincaré half-plane. Using~\eqref{eq:fr-distance-elliptical}, we get that the Fisher--Rao distance in this manifold is
\begin{equation}
	d_{\FR} \big( (\mu_1,\sigma_1), (\mu_2,\sigma_2) \big)
	= 2 \arctanh \left( \sqrt{\frac{(\mu_1-\mu_2)^2+(\sigma_1-\sigma_2)^2}{(\mu_1-\mu_2)^2+(\sigma_1+\sigma_2)^2}} \right).
\end{equation}

The particular case of zero-mean Laplace distributions is included in~\cite{verdoolaege2012}.

\subsubsection{Generalised Gaussian} \label{subsec:generalised-gaussian}

A generalised Gaussian distribution\footnote{
	These generalised Gaussian distributions are in a different sense that those considered in~\cite{andai2009}.
}~\cite{dytso2018,fang1990,verdoolaege2012}, also known as exponential power distribution, is characterised by the p.d.f. $p(x) = \frac{\beta}{2\sigma \Gamma(1/\beta)}\exp\left(-\frac{\vert x-\mu\vert^{\beta}}{\sigma}\right)$, defined for $x \in \R$, and parametrised by $(\mu, \sigma) \in \R \times \R_+^*$, for a fixed $\beta>0$. These can be seen as elliptical distributions with $h(u)=\frac{\beta}{2\Gamma(1/\beta)}\exp(-u^{\beta/2})$, $a_h = \beta\frac{\Gamma(2-1/\beta)}{\Gamma(1+1/\beta)}$, and $b_h=\beta$. Using~\eqref{eq:elliptical-metric}, we get the Fisher matrix as
\begin{equation} \label{eq:metrica-generalized-gaussiana}
	G(\mu,\sigma) = \begin{pmatrix}
		\frac{\beta}{\sigma^2}\frac{\Gamma(2-1/\beta)}{\Gamma(1+1/\beta)} & 0\\
		0 & \frac{\beta}{\sigma^2}
	\end{pmatrix},
\end{equation}
and, using~\eqref{eq:fr-distance-elliptical}, we get the Fisher--Rao distance as 
\begin{align}
	&\hspace{-2em}d_{\FR} \big((\mu_1,\sigma_1), (\mu_2, \sigma_2)\big)\nonumber\\
	&=
	\sqrt{\beta+1} \arctanh \left( \sqrt{\frac{\beta(\mu_1-\mu_2)^2\Gamma(2-1/\beta) + (\beta+1)(\sigma_1-\sigma_2)^2\Gamma(1+1/\beta)}{\beta(\mu_1-\mu_2)^2\Gamma(2-1/\beta) + (\beta+1)(\sigma_1+\sigma_2)^2\Gamma(1+1/\beta)}} \right).
\end{align}

Note that choosing $\beta=2$ yields a Gaussian distribution with mean $\mu$ and variance $\sigma^2/2$, while letting $\beta=1$ corresponds to a Laplace distribution with mean $\mu$ and variance $8\sigma^2$. Multivariate, zero-mean generalised Gaussian distributions have been studied in~\cite{verdoolaege2012}.

\subsubsection{Logistic} \label{subsec:logistic}

A logistic distribution~\cite{oller1987} has p.d.f. $p(x) = \frac{\exp\left(-(x-\mu)/\sigma\right)}{\sigma\left(\exp\left(-(x-\mu)/\sigma\right)+1\right) ^2}$, defined for $x \in \R$ and parametrised by $(\mu, \sigma) \in \R \times \R^*_+$. A logistic distribution is an elliptical distribution with $h(u)=\frac{\exp\left(-\sqrt{u}\right)}{\left(1+\exp\left(-\sqrt{u}\right)\right)^2}$, $a_h=\frac{1}{3}$, and $b_h=\frac{\pi^2+3}{9}$. From~\eqref{eq:elliptical-metric}, we have that the Fisher matrix is
\begin{equation} \label{eq:fisher-matrix-logistic}
	G(\mu,\sigma) = \begin{pmatrix}
		\frac{1}{3\sigma^2} & 0\\
		0 & \frac{\pi^2+3}{9\sigma^2}
	\end{pmatrix},
\end{equation}
and, from~\eqref{eq:fr-distance-elliptical}, the Fisher--Rao distance is
\begin{align}
	d_{\FR}\big( (\mu_1,\sigma_1), (\mu_2,\sigma_2) \big)
	&= \frac{2\sqrt{\pi^2+3}}{3} \arctanh \left( \sqrt{\frac{3(\mu_2-\mu_1)^2+(\pi^2+3)(\sigma_2-\sigma_1)^2}{3(\mu_2-\mu_1)^2+(\pi^2+3)(\sigma_2+\sigma_1)^2}} \right).
\end{align}

\subsubsection{Cauchy} \label{subsec:cauchy}

A Cauchy distribution~\cite{mithcell1988,nielsen2020-cauchy} has p.d.f. $p(x) = \frac{\sigma}{\pi \left[ (x-\mu)^2 + \sigma^2 \right]}$, defined for $x \in \R$, and parametrised by $(\mu, \sigma) \in \R \times \R_+^*$. Cauchy distributions are elliptical distributions by choosing $h(u)=\frac{1}{\pi(1+u)}$, $a_h=1/2$, $b_h=1/2$. Recall that the mean and variance are not defined in this case. From~\eqref{eq:elliptical-metric}, we get its the Fisher matrix
\begin{equation} \label{eq:fisher-matrix-cauchy}
	G(\mu,\sigma) =
	\begin{pmatrix}
		\frac{1}{2\sigma^2} & 0\\
		0 & \frac{1}{2\sigma^2}
	\end{pmatrix},
\end{equation}
and~\eqref{eq:fr-distance-elliptical} gives the Fisher--Rao distance as
\begin{equation}
	d_{\FR}\big( (\mu_1,\sigma_1), (\mu_2, \sigma_2) \big)
	= {\sqrt{2}} \arctanh \left( \sqrt{\frac{(\mu_1-\mu_2)^2+(\sigma_1-\sigma_2)^2}{(\mu_1-\mu_2)^2+(\sigma_1+\sigma_2)^2}} \right).
\end{equation}

\begin{figure}
	\centering
	\begin{subfigure}[b]{0.4\textwidth}
		\centering
		\includegraphics[height=0.1\textheight]{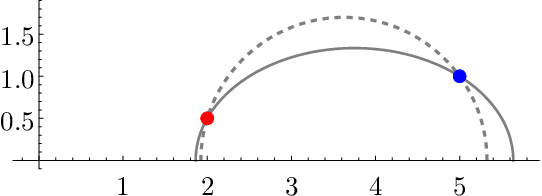}
		\caption{Geodesics in the parameter space $(\mu,\sigma)$.}
	\end{subfigure}
	\hspace{0.08\textwidth}
	\begin{subfigure}[b]{0.2\textwidth}
		\centering
		\includegraphics[height=0.1\textheight]{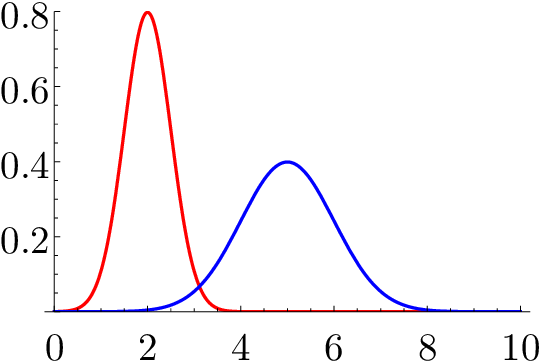}
		\caption{Gaussian densities.}
	\end{subfigure}
	\hspace{0.08\textwidth}
	\begin{subfigure}[b]{0.2\textwidth}
		\centering
		\includegraphics[height=0.1\textheight]{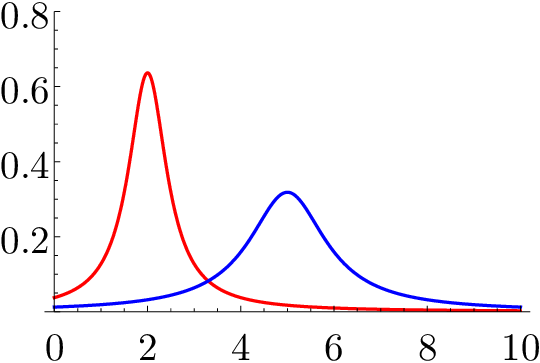}
		\caption{Cauchy densities.}
	\end{subfigure}
	\caption{
		Geodesics joining points $(\mu_1,\sigma_1)=(2,0.5)$ and $(\mu_2,\sigma_2)=(5,1)$ according to Gaussian metric (solid) and Cauchy metric (dashed), seen in the parameter space $(\mu,\sigma)$, and the corresponding densities. The distance between the two Gaussian distributions is $d_{\text{Gaussian}}((\mu_1,\sigma_1),(\mu_2,\sigma_2)) \approx 3.443$, and between the Cauchy distributions is $d_{\text{Cauchy}}((\mu_1,\sigma_1),(\mu_2,\sigma_2)) \approx 1.721$.
	}
	\label{fig:geodesics-elliptical}
\end{figure}

Figure~\ref{fig:geodesics-elliptical} illustrates geodesics between Gaussian and Cauchy distributions.

\subsubsection{Student's $t$} \label{subsec:student-t}

A location-scale Student's $t$ distribution~\cite{mithcell1988} with $\nu \in \N^*$ degrees of freedom generalises the  Cauchy distribution. It has p.d.f. $p(x) = \left(1+ \frac{1}{\nu}(\frac{x-\mu}{\sigma})^2 \right)^{-\frac{\nu+1}{2}} \frac{\Gamma\left((\nu+1)/2\right)}{\sigma\sqrt{\pi\nu}\Gamma(\nu/2)}$, defined for $x \in \R$, and parametrised by $(\mu,\sigma) \in \R \times \R_+^*$. This is an elliptical distribution, with $h(u)=\frac{\Gamma\left((\nu+1)/2\right)}{\sqrt{\pi\nu}\Gamma(\nu/2)} \left(1+\frac{u}{\nu}\right)^{-(\nu+1)/2}$, $a_h=\frac{\nu+1}{\nu+3}$, and $b_h=\frac{2\nu}{\nu+3}$. Then, by~\eqref{eq:elliptical-metric} we obtain the Fisher matrix
\begin{equation}\label{eq:fisher-matrix-student-t}
	G(\mu,\sigma) = 
	\begin{pmatrix} 
		\frac{\nu+1}{(\nu+3)\sigma^2} & 0\\
		0 & \frac{2\nu}{(\nu+3)\sigma^2}
	\end{pmatrix},
\end{equation}
and by~\eqref{eq:fr-distance-elliptical}, the Fisher--Rao distance
\begin{equation}
	d_{\FR}\big( (\mu_1,\sigma_1), (\mu_2,\sigma_2) \big)
	=2\sqrt{\frac{2\nu}{\nu+3}} \arctanh \left( \sqrt{\frac{(\nu+1)(\mu_2-\mu_1)^2+2\nu(\sigma_2-\sigma_1)^2}{(\nu+1)(\mu_2-\mu_1)^2+2\nu(\sigma_2+\sigma_1)^2}} \right).
\end{equation}

\begin{remark}
	We close this subsection with a remark on the general case of multivariate elliptical distributions. These are distributions of the form $p(\xv) = (\det \Sigma)^{-1/2} h \left( (\pmb{x}-\mub)^\t \Sigma^{-1} (\pmb{x} - \mub) \right)$, for some function~$h \colon \R_+ \to \R_+$, defined for $\xv \in \R^n$, and characterised by a vector $\mub \in \R^n$ and an $n \times n$ positive-definite symmetric matrix $\Sigma \in P_n(\R)$. Analogously to the univariate case, the set of elliptical distributions for a fixed~$h$ forms an $\left(n + \frac{n(n+1)}{2}\right)$-dimensional statistical manifold. In the general case, however, no general closed-form expression for the Fisher--Rao distance is known; instead, only expressions for particular cases and bounds for this distance are available~\cite{berkane1997,calvo2002,chen2021,mithcell1988,mitchell1985}.
	Multivariate Gaussian distributions $p(\pmb{x}) =  \left((2\pi)^n \det \Sigma\right)^{-1/2} \exp \left( - \frac{1}{2} (\pmb{x}-\mub)^\t \Sigma^{-1} (\pmb{x} - \mub) \right)$ have been particularly studied~\cite{atkinson1981,burbea1986,calvo1990,calvo1991,krzanowski1996,pinele2020,skovgaard1984}. Special cases for which the Fisher--Rao distance can be written include: fixed mean, fixed covariance, diagonal covariance matrix, and mirrored distributions. The case of diagonal covariance matrix is equivalent to independent components and will be treated in Section~\ref{sec:product-distributions}. Special cases for the multivariate generalised Gaussian distributions have been studied in~\cite{verdoolaege2012}. The Fisher--Rao distance between zero-mean complex elliptically symmetric distributions has been computed in~\cite{bouchard2023,breloy2019}.
\end{remark}

\subsection{Log-Gaussian} \label{subsec:log-gaussian}

A log-Gaussian distribution~\cite{arwini2008,calin2014} is the distribution of a random variable whose logarithm follows a Gaussian distribution. It has p.d.f. $p(x) = \frac{1}{\sigma x\sqrt{2\pi}} \exp \left( - \frac{(\log x - \mu)^2}{2\sigma^2} \right)$, defined for $x \in \R_+^*$, and parametrised by $(\mu,\sigma) \in \R \times \R_+^*$. In this case, we have $\partial_{\mu} \ell \coloneqq \partial_{\mu} \ell(\mu,\sigma) = \frac{\log x - \mu}{\sigma^2}$ and $\partial_{\sigma} \ell \coloneqq \partial_{\sigma} \ell(\mu,\sigma) = -\frac{1}{\sigma} + \frac{(\log x-\mu)^2}{\sigma^3}$, so that the elements of the Fisher matrix are
\begin{align*}
	g_{11} &= \E \left[ \left( \partial_{\mu}\ell \right)^2 \right]
	= \E \left[ \left( \frac{\log X-\mu}{\sigma^2} \right)^2 \right]\\
	&= \frac{\E[(\log X-\mu)^2]}{\sigma^4}
	= \frac{1}{\sigma^2},\\
	g_{12} = g_{21} &= \E \left[ \left(\partial_\mu \ell\right) \left(\partial_\sigma \ell\right) \right]
	= \E \left[ \left( \frac{\log X-\mu}{\sigma^2} \right) \left( -\frac{1}{\sigma} + \frac{(\log X-\mu)^2}{\sigma^3} \right) \right]\\
	&= -\frac{\E[\log X-\mu]}{\sigma^3} + \frac{\E[(\log X-\mu)^3]}{\sigma^5}\\
	&= 0,\\
	g_{22} &= \E \left[ \left( \partial_\sigma \ell \right)^2 \right]
	= \E \left[ \left( -\frac{1}{\sigma} + \frac{(\log X-\mu)^2}{\sigma^3} \right)^2 \right]\\
	&= \frac{1}{\sigma^2} - \frac{2 \E[(\log X-\mu)^2]}{\sigma^4} + \frac{\E[(\log X-\mu)^4]}{\sigma^6}
	= \frac{2}{\sigma^2},
\end{align*}

\noindent having used that $\E[\log X] = \mu$, $\E[\log X - \mu] = \E[(\log X - \mu)^3] = 0$, $\E[(\log X-\mu)^2] = \sigma^2$, and $\E[(\log X-\mu)^4]=3\sigma^4$. Thus, the Fisher matrix is given by
\begin{equation} \label{eq:metrica-log-gaussiana}
	G(\mu,\sigma) = \begin{pmatrix}
		\frac{1}{\sigma^2} & 0\\
		0 & \frac{2}{\sigma^2}
	\end{pmatrix}.
\end{equation}

\noindent This is the same as for the Gaussian manifold~\eqref{eq:fisher-matrix-gaussian}, therefore the Fisher--Rao distance the same, namely,
\begin{equation}
	d_\FR \big((\mu_1,\sigma_1), (\mu_2, \sigma_2)\big)
	= 2\sqrt{2} \arctanh \left( \sqrt{\frac{(\mu_1-\mu_2)^2 + 2(\sigma_1-\sigma_2)^2}{(\mu_1-\mu_2)^2 + 2(\sigma_1+\sigma_2)^2}} \right).
\end{equation}

\subsection{Inverse Gaussian}

An inverse Gaussian distribution ~\cite{khan2022,villarroya1993} has p.d.f. $p(x) = \sqrt{\frac{\lambda}{2\pi x^3}} \exp \left( - \frac{\lambda(x - \mu)^2}{2\mu^2 x} \right)$, defined for $x \in \R_+^*$, and parametrised by $(\lambda,\mu) \in \R_+^* \times \R_+^*$. In this case, we have $\partial_{\lambda} \ell \coloneqq \partial_{\lambda} \ell(\lambda,\mu) = \frac{1}{2\lambda} - \frac{(x-\mu)^2}{2\mu^2 x}$ and $\partial_{\mu} \ell \coloneqq \partial_{\mu} \ell(\lambda,\mu) = \frac{\lambda( x - \mu)}{\mu^3}$. The elements of the Fisher matrix are then
\begin{align*}
	g_{11} &= \E \left[ \left( \partial_{\lambda}\ell \right)^2 \right]
	=\E \left[ \left( \frac{1}{2\lambda} - \frac{(x-\mu)^2}{2\mu^2 x}\right)^2 \right] \\
	&= \frac{1}{4\lambda^2}+\frac{1}{\lambda\mu}+\frac{3}{2\mu^2}-\left(\frac{1}{2\lambda\mu^2}+\frac{1}{\mu^3}\right)\E[X]\nonumber\\
	&\quad¨-\left(\frac{1}{2\lambda}+\frac{1}{\mu}\right)\E\left[\frac{1}{X}\right]+\frac{1}{4\mu^4}\E\left[X^2\right]+\frac{1}{4}\E\left[\frac{1}{X^2}\right]\\
	&= \frac{1}{2\lambda^2},\\
	g_{12} = g_{21} &= \E \left[ \left(\partial_{\lambda} \ell\right) \left(\partial_{\mu} \ell\right) \right]
	= \E \left[ \left(\frac{1}{2\lambda} - \frac{(x-\mu)^2}{2\mu^2 x}\right)\left( \frac{\lambda( x - \mu)}{\mu^3} \right)  \right]\\
	&= -\frac{1}{2\lambda\mu^2}-\frac{3}{2\mu^3}+\left(\frac{1}{2\lambda\mu^3}-\frac{3\lambda}{2\mu^4}\right)\E[X]+\frac{\lambda}{2\mu^2}\E\left[\frac{1}{X}\right]-\frac{\lambda}{2\mu^5}\E\left[X^2\right]\\
	&= 0,\\
	g_{22} &= \E \left[ \left( \partial_{\mu} \ell \right)^2 \right]
	= \E \left[ \left( \frac{\lambda( x - \mu)}{\mu^3} \right)^2 \right] \\
	&= \frac{\lambda^2}{\mu^4}-\frac{2\lambda^2}{\mu^5}\E[X]+\frac{\lambda^2}{\mu^6}\E\left[X^2\right]\\
	&= \frac{\lambda}{\mu^3},
\end{align*}

\noindent where we have used that $\E[X] = \mu$, $\E\left[X^2\right] = \frac{\mu^3}{\lambda}+\mu^2$, $\E\left[\frac{1}{X}\right] = \frac{1}{\lambda}+\frac{1}{\mu}$ and $\E\left[\frac{1}{X^2} \right] = \frac{3}{\lambda^2}+\frac{3}{\lambda\mu}+\frac{1}{\mu^2}$. Thus, the Fisher matrix is given by
\begin{equation} \label{eq:fisher-matrix-inverse-gaussian}
	G(\lambda,\mu) = \begin{pmatrix}
		\frac{1}{2\lambda^2} & 0\\
		0 & \frac{\lambda}{\mu^3}
	\end{pmatrix}.
\end{equation}

To find the Fisher--Rao distance, we consider the change of coordinates $u=1/\sqrt{\lambda}$, $v=\sqrt{2/\mu}$. Applying Proposition~\ref{prop:reparametrisation-parameter-space} we find that the Fisher matrix in the new coordinates
\begin{equation*}
	\widetilde{G}(u,v) =
	\begin{pmatrix}
		\frac{2}{u^2} & 0\\
		0 & \frac{2}{u^2}
	\end{pmatrix}
	=2 \begin{pmatrix}
		\frac{1}{u^2} & 0\\
		0 & \frac{1}{u^2}
	\end{pmatrix}.
\end{equation*}
Applying Lemma~\ref{lemma:poincare-isometry} yields
\begin{align}
	d_{\FR}\big( (\lambda_1,\mu_1), (\lambda_2,\mu_2) \big)
	&= \sqrt{2}d_{\Hcal^2} \left( \left(1/\sqrt{\lambda_1}, \sqrt{2/\mu_1}\right), \left(1/\sqrt{\lambda_2},\sqrt{2/\mu_2}\right)\right)\nonumber\\
	&= 2\sqrt{2} \arctanh \left(
	\sqrt{\frac{
			\mu_1\mu_2(\sqrt{\lambda_1}-\sqrt{\lambda_2})^2+2\lambda_1\lambda_2(\sqrt{\mu_1}-\sqrt{\mu_2})^2
		}{
			\mu_1\mu_2(\sqrt{\lambda_1}-\sqrt{\lambda_2})^2+2\lambda_1\lambda_2(\sqrt{\mu_1}+\sqrt{\mu_2})^2}
	}
	\right).
\end{align}

\subsection{Extreme-value distributions}

Extreme-value distributions~\cite{kotz2000} are limit distribution for the maxima (or minima) of a sequence of i.i.d. random variables. They are usually considered to be of one of three families, all of which can be described in the form of the generalised extreme-value distributions:
\begin{equation} \label{eq:generalised-extreme-value}
	p(x) = \frac{1}{\sigma} {[t(x,\xi)]}^{\xi+1} e^{-t(x,\xi)},
\end{equation}
where $\sigma>0$, $\xi\in\R$, and
\begin{equation*}
	t(x, \xi) =
	\begin{cases}
		\left( 1 + \xi \left(\frac{x-\mu}{\sigma}\right) \right)^{-\frac{1}{\xi}}, &\xi\neq0,\\
		\exp\left(-\frac{x-\mu}{\sigma}\right), &\xi=0,
	\end{cases}
\end{equation*}
for $\mu\in\R$. Note that $t(x,\xi)$ is continuous on $\xi=0$, for every $x\in\R$. When $\xi=0$, the support of \eqref{eq:generalised-extreme-value} is $x\in\R$, and it is called type I or \emph{Gumbel-type} distribution; when $\xi>0$, the support is $x \in \left[\mu-\frac{\sigma}{\xi},\ +\infty\right[$, and it is called type II or \emph{Fréchet-type} distribution; when $\xi<0$, the support is $x \in \left]-\infty,\ \mu-\frac{\sigma}{\xi}\right]$, and it is called type III or \emph{Weibull-type} distribution. Therefore, in general, these distributions are parametrised by the triple $(\mu,\sigma,\xi) \in \R \times \R^*_+ \times \R$. Instead of treating the three-dimensional manifold of generalised extreme-value distributions in full generality~\cite{prescott1980,taylor2019}, for which no closed-form expression for the Fisher--Rao is available, we consider the usual two-dimensional versions of them, following~\cite{oller1987}.

\subsubsection{Gumbel} \label{subsec:gumbel}

A Gumbel distribution~\cite{oller1987} has p.d.f. $p(x) = \frac{1}{\sigma} \exp\left(-\frac{x-\mu}{\sigma}\right) \exp\left(-\exp\left(-\frac{x-\mu}{\sigma}\right)\right) $, defined for $x \in \R$, and parametrised by $(\mu, \sigma) \in \R \times \R^*_+$. This corresponds to~\eqref{eq:generalised-extreme-value} taking $\xi=0$. In this case, we have $\partial_{\mu} \ell \coloneqq \partial_{\mu} \ell(\mu,\sigma) = \frac{1}{\sigma} -\frac{1}{\sigma}\exp\left(-\frac{x-\mu}{\sigma}\right) $ and $\partial_{\sigma} \ell \coloneqq \partial_{\sigma} \ell(\mu,\sigma) =- \frac{x-\mu}{\sigma^2}\exp\left(-\frac{x-\mu}{\sigma}\right)+\frac{x-\mu}{\sigma^2}-\frac{1}{\sigma}$. Denoting $ Z \coloneq \frac{X-\mu}{\sigma}$, the elements of the Fisher matrix can be written as
\begin{align*}
	g_{11}
	&= \E \left[ \left( \partial_{\mu}\ell \right)^2 \right]
	= \E \left[ \left( \frac{1}{\sigma} -\frac{1}{\sigma}e^{-Z} \right)^2 \right]\\
	&= \frac{1}{\sigma^2} - \frac{2}{\sigma^2}\E\left[e^{-Z}\right] + \frac{1}{\sigma^2}\E\left[e^{-2Z}\right]\\
	&= \frac{1}{\sigma^2},\\
	g_{12} = g_{21}
	&= \E \left[ \left(\partial_{\mu} \ell\right) \left(\partial_{\sigma} \ell\right) \right]\\
	&= \E \left[ \left(\frac{1}{\sigma} -\frac{1}{\sigma} e^{-Z} \right) \left( \frac{Ze^{-Z}}{\sigma} + \frac{Z}{\sigma}+\frac{1}{\sigma^2}\right)\right]\\
	&= -\frac{2}{\sigma^2} \E \left[Ze^{-Z}\right]+\frac{1}{\sigma^2} \E \left[Z\right]-\frac{1}{\sigma^2} +\frac{1}{\sigma^2}\E \left[Ze^{-2Z}\right]+\frac{1}{\sigma^2} \E \left[e^{-Z}\right]\\
	&=\frac{\gamma -1}{\sigma^2},\\
	g_{22}
	&= \E \left[ \left( \partial_{\sigma} \ell \right)^2 \right]
	= \E \left[ \left( \frac{Ze^{-Z}}{\sigma} + \frac{Z}{\sigma} + \frac{1}{\sigma^2}\right)^2 \right]\\
	&= \frac{1}{\sigma^2}\E \left[Z^2 e^{-2Z}\right] - \frac{2}{\sigma^2}\E \left[Z^2 e^{-Z}\right]+\frac{2}{\sigma^2}\E \left[Ze^{-Z}\right] +\frac{1}{\sigma^2}\E \left[Z^2\right]-\frac{2}{\sigma^2}\E \left[Z\right]+\frac{1}{\sigma^2}\\
	&= \frac{1}{\sigma^2}\left((\gamma-1)^2+\frac{\pi^2}{6}\right),
\end{align*}

\noindent where $\gamma$ is the Euler constant and we have used that $\E \left[Z\right]=\gamma$, $\E \left[e^{-Z}\right]=1$, $\E \left[e^{-2Z}\right]=2$, $\E \left[Ze^{-Z}\right]=\gamma-1$, $\E \left[Ze^{-2Z}\right]=2\gamma-3$, $\E \left[Z^2e^{-2Z}\right]=2\gamma^2-6\gamma+2+\frac{\pi^2}{3}$, $\E \left[Z^2e^{-Z}\right]=\gamma^2-2\gamma+\frac{\pi^2}{6}$ and $\E \left[Z^2\right]=\gamma^2+\frac{\pi^2}{6}$. Thus, the Fisher matrix is
\begin{equation} \label{eq:fisher-matrix-gumbel}
	G(\mu,\sigma) = \begin{pmatrix}
		\frac{1}{\sigma^2} & \frac{\gamma-1}{\sigma^2}\\
		\frac{\gamma-1}{\sigma^2} & \frac{1}{\sigma^2}\left((\gamma-1)^2+\frac{\pi^2}{6}\right)
	\end{pmatrix}.
\end{equation}

To find the Fisher--Rao distance, we consider the change of coordinates $u=\mu-(1-\gamma)\sigma$, $v=\pi\sigma/\sqrt{6}$. Applying Proposition~\ref{prop:reparametrisation-parameter-space} we find that the Fisher matrix in the new coordinates is
\begin{equation*}
	\widetilde{G}(u,v) =
	\begin{pmatrix}
		\frac{\pi^2}{6v^2} & 0\\
		0 & \frac{\pi^2}{6v^2}
	\end{pmatrix}
	=\frac{\pi^2}{6} \begin{pmatrix}
		\frac{1}{v^2} & 0\\
		0 & \frac{1}{v^2}.
	\end{pmatrix}
\end{equation*}

\noindent Then, applying Lemma~\ref{lemma:poincare-isometry}, the Fisher--Rao distance is given by
\begin{align} \label{eq:fr-distance-gumbel}
	&d_{\FR}\big( (\mu_1,\sigma_1), (\mu_2,\sigma_2) \big)\nonumber\\
	&= \frac{\pi}{\sqrt{6}}d_{\Hcal^2} \left(\left( \mu_1-(1-\gamma)\sigma_1, \frac{\pi}{\sqrt{6}}\sigma_1\right), \left(\mu_2-(1-\gamma)\sigma_2, \frac{\pi}{\sqrt{6}}\sigma_2\right)\right)\nonumber\\
	&= \frac{2\pi}{\sqrt{6}} \arctanh \left( \sqrt{
		\frac{
			\left[(\mu_1-\mu_2)-(1-\gamma)(\sigma_1-\sigma_2)\right]^2+\frac{\pi^2}{6}(\sigma_1-\sigma_2)^2
		}{
			\left[(\mu_1-\mu_2)-(1-\gamma)(\sigma_1-\sigma_2)\right]^2+\frac{\pi^2}{6}(\sigma_1+\sigma_2)^2
		}
	} \right).
\end{align}

\subsubsection{Fréchet} \label{subsec:frechet}

A Fréchet distribution~\cite{oller1987} has p.d.f. $p(x) = \frac{\lambda}{\beta} \left( \frac{x}{\beta} \right)^{-\lambda-1} \exp \left( - \left(\frac{x}{\beta}\right)^{-\lambda} \right)$, defined for $x \in \R_+^*$, and parametrised by $(\beta,\lambda) \in \R_+^* \times \R_+^*$. Note that this corresponds to~\eqref{eq:generalised-extreme-value} taking $\mu = \beta$, $\sigma = \beta/\lambda$ and $\xi = 1/\lambda$. This distribution can be related to the Gumbel distribution by considering the reparametrisation of the sample space $Y \coloneqq \log X$, which preserves the Fisher metric (Proposition~\ref{prop:reparametrisation-sample-space}). The p.d.f. of the new random variable is then
\begin{align*}
	p(y)
	= \frac{1}{\left|\frac{\d y}{\d x}\right|} \frac{\lambda}{\beta} \left( \frac{e^y}{\beta} \right)^{-\lambda-1} \exp \left( -\left( \frac{e^y}{\beta} \right)^{-\lambda} \right)
	= \lambda \left(\frac{e^y}{\beta}\right)^{-\lambda} \exp\left( - \left(\frac{e^y}{\beta}\right)^{-\lambda} \right).
\end{align*}

\noindent Now, considering the change of coordinates $\alpha = \log \beta$, $\theta = 1/\lambda$, we find
\begin{equation*}
	p(y) = \frac{1}{\theta} \exp\left( - \frac{y-\alpha}{\theta} \right) \exp \left( - \left( -\frac{y-\alpha}{\theta} \right) \right),
\end{equation*}
which coincides with that of a Gumbel distribution with parameters $(\alpha,\theta) \in \R \times \R^*_+$. Comparing with the Fisher matrix~\eqref{eq:fisher-matrix-gumbel} and  applying Proposition~\ref{prop:reparametrisation-parameter-space} we find that the Fisher matrix in the $(\beta,\lambda)$ coordinates is
\begin{equation} \label{eq:fisher-matrix-frechet}
	G(\beta,\lambda) =
	\begin{pmatrix}
		\frac{\lambda^2}{\beta^2} & \frac{1-\gamma}{\beta}\\
		\frac{1-\gamma}{\beta} & \frac{1}{\lambda^2}\left((\gamma-1)^2+\frac{\pi^2}{6}\right)
	\end{pmatrix}.
\end{equation}

Note that, by considering the change of coordinates $u=\log\beta-(1-\gamma)/\lambda$, $v=\pi/(\lambda\sqrt{6})$ and applying again Proposition~\ref{prop:reparametrisation-parameter-space}, the Fisher matrix is found to be
\begin{equation*}
	\widetilde{G}(u,v) =
	\begin{pmatrix}
		\frac{\pi^2}{6v^2} & 0\\
		0 & \frac{\pi^2}{6v^2}
	\end{pmatrix}
	=\frac{\pi^2}{6} \begin{pmatrix}
		\frac{1}{v^2} & 0\\
		0 & \frac{1}{v^2}.
	\end{pmatrix},
\end{equation*}

\noindent Therefore, by Lemma~\ref{lemma:poincare-isometry} the Fisher--Rao distance for Fréchet distributions is

\begin{align}
	&d_{\FR}\big( (\beta_1,\lambda_1), (\beta_2,\lambda_2) \big)	\nonumber\\
	&= \frac{\pi}{\sqrt{6}}d_{\Hcal^2}\left( \left(\log\beta_1-\frac{(1-\gamma)}{\lambda_1 }, \frac{\pi}{\lambda_1\sqrt{6}}\right), \left(\log\beta_2-\frac{(1-\gamma)}{\lambda_2 }, \frac{\pi}{\lambda_2\sqrt{6}}\right)\right)\nonumber\\
	&= \frac{2\pi}{\sqrt{6}} \arctanh \left( \sqrt{
		\frac{
			\left[\log\frac{\beta_1}{\beta_2}-(1-\gamma)\left(\frac{1}{\lambda_1}-\frac{1}{\lambda_2}\right)\right]^2+\frac{\pi^2}{6}\left(\frac{1}{\lambda_1}-\frac{1}{\lambda_2}\right)^2
		}{
			\left[\log\frac{\beta_1}{\beta_2}-(1-\gamma)\left(\frac{1}{\lambda_1}-\frac{1}{\lambda_2}\right)\right]^2+\frac{\pi^2}{6}\left(\frac{1}{\lambda_1}+\frac{1}{\lambda_2}\right)^2
		}
	} \right).
\end{align}

\subsubsection{Weibull} \label{subsec:weibull}

A Weibull distribution~\cite{burbea1986,oller1987,wauters1993} has p.d.f. $p(x) = \frac{\lambda}{\beta} \left( \frac{x}{\beta} \right)^{\lambda-1} \exp \left( - \left(\frac{x}{\beta}\right)^{\lambda} \right)$, defined for $x \in \R_+^*$, and parametrised by $(\beta,\lambda) \in \R_+^* \times \R_+^*$. Note that this corresponds to the distribution of $-X$ in~\eqref{eq:generalised-extreme-value} taking $\mu = -\beta$, $\sigma = \beta/\lambda$ and $\xi = -1/\lambda$. It is possible to relate a Weibull distribution to Gumbel by considering the reparametrisation of the sample space $Y \coloneqq - \log X$, which preserves the Fisher metric, (Proposition~\ref{prop:reparametrisation-sample-space}). The p.d.f. of the new random variable is
\begin{align*}
	p(y)
	= \frac{1}{\left|\frac{\d y}{\d x}\right|} \frac{\lambda}{\beta} \left( \frac{e^{-y}}{\beta} \right)^{\lambda-1} \exp \left( -\left( \frac{e^{-y}}{\beta} \right)^{\lambda} \right)
	= \lambda \left(\frac{e^{-y}}{\beta}\right)^{\lambda} \exp\left( - \left(\frac{e^{-y}}{\beta}\right)^{\lambda} \right).
\end{align*}

\noindent Moreover, with the change of coordinates $\lambda = 1/\theta$, $\alpha = -\log \beta$, we have
\begin{equation*}
	p(y) = \frac{1}{\theta} \exp\left( - \frac{y-\alpha}{\theta} \right) \exp \left( - \exp \left( -\frac{y-\alpha}{\theta} \right) \right),
\end{equation*}
which coincides with a a Gumbel distribution with parameters $(\alpha,\theta) \in \R \times \R^*_+$. Again, comparing with the Fisher matrix~\eqref{eq:fisher-matrix-gumbel} and  applying Proposition~\ref{prop:reparametrisation-parameter-space} we find that the Fisher matrix in the $(\beta,\lambda)$ coordinates is
\begin{equation} \label{eq:fisher-matrix-weibull}
	G(\beta,\lambda) =
	\begin{pmatrix}
		\frac{\lambda^2}{\beta^2} & \frac{\gamma-1}{\beta}\\
		\frac{\gamma-1}{\beta} & \frac{1}{\lambda^2}\left((\gamma-1)^2+\frac{\pi^2}{6}\right)
	\end{pmatrix}.
\end{equation}

\noindent The change of coordinates $u=-\log\beta-(1-\gamma)/\lambda$, $v=\pi/(\lambda\sqrt{6})$ with Proposition~\ref{prop:reparametrisation-parameter-space} yields the following form for the Fisher matrix:
\begin{equation*}
	\widetilde{G}(u,v) =
	\begin{pmatrix}
		\frac{\pi^2}{6v^2} & 0\\
		0 & \frac{\pi^2}{6v^2}
	\end{pmatrix}
	=\frac{\pi^2}{6} \begin{pmatrix}
		\frac{1}{v^2} & 0\\
		0 & \frac{1}{v^2}
	\end{pmatrix}.
\end{equation*}

\noindent In addition, by Lemma~\ref{lemma:poincare-isometry}, the Fisher--Rao distance for Weibull distributions is
\begin{align}
	&d_{\FR}\big( (\beta_1,\lambda_1), (\beta_2,\lambda_2) \big)	\nonumber\\
	&= \frac{\pi}{\sqrt{6}}d_{\Hcal^2} \left( \left(-\log\beta_1-\frac{(1-\gamma)}{\lambda_1 }, \frac{\pi}{\lambda_1\sqrt{6}}\right), \left(-\log\beta_2-\frac{(1-\gamma)}{\lambda_2 }, \frac{\pi}{\lambda_2\sqrt{6}}\right)\right)\nonumber\\
	&= \frac{2\pi}{\sqrt{6}} \arctanh \left( \sqrt{
		\frac{
			\left[\log\frac{\beta_2}{\beta_1}-(1-\gamma)\left(\frac{1}{\lambda_1}-\frac{1}{\lambda_2}\right)\right]^2+\frac{\pi^2}{6}\left(\frac{1}{\lambda_1}-\frac{1}{\lambda_2}\right)^2
		}{
			\left[\log\frac{\beta_2}{\beta_1}-(1-\gamma)\left(\frac{1}{\lambda_1}-\frac{1}{\lambda_2}\right)\right]^2+\frac{\pi^2}{6}\left(\frac{1}{\lambda_1}+\frac{1}{\lambda_2}\right)^2
		}
	} \right).
\end{align}

The special case of fixed $\lambda$ has been addressed in~\cite{burbea1986}.

\begin{remark}
	If $X$ is a random variable following a Weibull distribution, then $-X$ follows a reversed Weibull distribution, which corresponds to the Weibull-type distribution from~\eqref{eq:generalised-extreme-value}, and has the same geometry, and same Fisher--Rao distance as the Weibull distribution~\cite{oller1987}.
\end{remark}

\subsection{Pareto} \label{subsec:pareto}

A Pareto distribution~\cite{burbea1986,li2022} has p.d.f. $p(x) = {\theta}\alpha^\theta{x^{-(\theta+1)}}$, defined for $x \in \left[\alpha,\infty\right[$ and parametrised by $(\theta, \alpha) \in \R^*_+ \times \R^*_+$.
In this case, the support depends on the parametrisation, thus violating one of the assumptions made in the definition of a statistical manifold~\eqref{eq:statistical-model}. Nevertheless, it is still possible\footnote{
	As noted in~\cite{li2022}, what is not possible is to use the alternative expression~\eqref{eq:fisher-matrix-hessian}, which would result in a `fake metric'.
} to compute a Riemannian metric from the Fisher information matrix, as in~\eqref{eq:fisher-matrix-gij}. We thus have $\partial_{\theta} \ell \coloneqq \partial_{\theta} \ell(\theta,\alpha) = \frac{1}{\theta} + \log \alpha - \log x$ and $\partial_{\alpha} \ell \coloneqq \partial_{\alpha} \ell(\theta,\alpha) = \frac{\theta}{\alpha}$.
The elements of the Fisher matrix are
\begin{align*}
	g_{11}
	&= \E \left[ \left( \partial_{\theta}\ell \right)^2 \right]
	= \E \left[ \left( \frac{1}{\theta} + \log \alpha - \log X \right)^2 \right]\\
	&= \frac{1}{\theta^2} + (\log \alpha)^2 + \E\left[ (\log X)^2 \right] + \frac{2}{\theta}\left( \log \alpha - \E[\log X] \right) - 2 \log \alpha \E[\log X]\\
	&= \frac{1}{\theta^2},\\
	g_{12} = g_{21}
	&= \E \left[ \left(\partial_\theta \ell\right) \left(\partial_{\alpha} \ell\right) \right]
	= \E \left[ \left( \frac{1}{\theta} + \log \alpha - \log X \right) \left( \frac{\theta}{\alpha} \right) \right]\\
	&= \frac{1}{\alpha} + \frac{\theta}{\alpha}\log \alpha - \frac{\theta}{\alpha} \E[\log X]
	= 0,\\
	g_{22}
	&= \E \left[ \left( \partial_{\alpha} \ell \right)^2 \right]
	= \E \left[ \left( \frac{\theta}{\alpha} \right)^2 \right]\\
	&= \frac{\theta^2}{\alpha^2},
\end{align*}

\noindent where we have used that $\E[\log X] = \frac{1}{\theta} + \log \alpha$ and $\E[(\log X)^2] = \frac{2}{\theta^2} + \frac{2 \log \alpha}{\theta} + (\log \alpha)^2$. Thus, the Fisher matrix is
\begin{equation} \label{eq:fisher-matrix-pareto}
	G(\theta,\alpha) = \begin{pmatrix}
		\frac{1}{\theta^2} & 0\\
		0 & \frac{\theta^2}{\alpha^2}
	\end{pmatrix}.
\end{equation}

To find the Fisher--Rao distance, we consider the change of coordinates $u=\log \alpha$, $v=1/\theta$. Applying Proposition~\ref{prop:reparametrisation-parameter-space} we find that the Fisher matrix in the new coordinates
\begin{equation*}
	\widetilde{G}(u,v)=
	\begin{pmatrix}
		\frac{1}{v^2} & 0\\
		0 & \frac{1}{v^2}
	\end{pmatrix},
\end{equation*}
which coincides with the hyperbolic metric~\eqref{eq:hyperbolic-metric} restricted to the positive quadrant. Therefore the Fisher--Rao distance is given by
\begin{align} \label{eq:distance-pareto}
	d_{\FR}\big( (\theta_1,\alpha_1), (\theta_2, \alpha_2) \big)
	&= d_{\Hcal^2} \left( \left(\log \alpha_1, 1/{\theta_1}\right), \left(\log \alpha_2, 1/{\theta_2} \right) \right)\nonumber\\
	&= 2 \arctanh \left( \sqrt{\frac{(\theta_1\theta_2\log(\alpha_1/\alpha_2))^2+(\theta_1-\theta_2)^2}{(\theta_1\theta_2\log(\alpha_1/\alpha_2))^2+(\theta_1+\theta_2)^2}} \right).
\end{align}

The special case of fixed $\alpha$ has been addressed in~\cite{burbea1986}.

\subsection{Power function} \label{subsec:power-function}

A power function distribution~\cite{burbea1986} has p.d.f. $p(x) = \theta \beta^{-\theta} x^{\theta-1}$, defined for $x \in \left]0,\beta\right]$, and parametrised by $(\theta,\beta) \in \R_+^* \times \R_+^*$. As in the previous example, the support depends on the parametrisation, but it is still possible to consider the Fisher metric as in~\eqref{eq:fisher-matrix-gij}. This distribution can be related to the Pareto distribution (\S~\ref{subsec:pareto}) as follows. Consider the reparametrisation of the sample space given by $Y \coloneqq 1/X$ (cf. Proposition~\ref{prop:reparametrisation-sample-space}), and the change of coordinates $\alpha = 1/\beta$. Note that, since $x \in \left]0,\beta\right]$, we have $y \in \left[\alpha,\infty\right[$.  The p.d.f. of the new random variable, with the new coordinates, is
\begin{equation*}
	p(y) = \frac{1}{\left| \frac{\d y}{\d x} \right|} \theta \beta^{-\theta} {y}^{-(\theta-1)} = \theta \beta^{-\theta} y^{-(\theta+1)} = \theta \alpha^{\theta} y^{-(\theta+1)}.
\end{equation*}
which coincides with a Pareto distribution with parameters $(\theta,\alpha)$. Therefore, applying Proposition~\ref{prop:reparametrisation-parameter-space}, we find
\begin{equation}
	G(\theta,\beta) = \begin{pmatrix}
		\frac{1}{\theta^2} & 0\\
		0 & \frac{\theta^2}{\beta^2}
	\end{pmatrix},
\end{equation}
and therefore
\begin{align}
	d_{\FR}\big( (\theta_1,\beta_1), (\theta_2, \beta_2) \big)
	&= d_{\Hcal^2} \left( \left(\log \beta_1, 1/{\theta_1} \right), \left( \log \beta_2, 1/{\theta_2} \right) \right)\nonumber\\
	&= 2 \arctanh \left( \sqrt{\frac{(\theta_1\theta_2\log(\beta_{1}/\beta_{2}))^2+(\theta_1-\theta_2)^2}{(\theta_1\theta_2\log(\beta_{1}/\beta_{2}))^2+(\theta_1+\theta_2)^2}} \right).
\end{align}

The special case of fixed $\alpha$ has been addressed in~\cite{burbea1986}.

\begin{remark}
	All the examples of two-dimensional statistical manifolds of continuous distributions presented so far in this section are related to the hyperbolic Poincaré half-plane, and have constant negative curvature. However, there are examples of two-dimensional statistical manifolds which are not of constant negative curvature, even if we do not have an explicit expression for the Fisher--Rao distances. We present some examples in the following.
	\begin{itemize}
		\item The statistical manifold of Gamma distributions~ $p(x) = \frac{\beta^{\alpha}}{\Gamma(\alpha)}x^{\alpha-1}e^{-\beta x}$, defined for $x \in \R_+^*$ and parametrised by $(\alpha,\beta) \in \R_+^* \times \R_+^*$. The curvature of this manifold is~\cite{burbea2002,lauritzen1987,reverter2003}
		\begin{equation*}
			\kappa = \frac{\psi^{(1)}(\alpha) + \alpha \psi^{(2)}(\alpha)}{4\left(\alpha \psi^{(1)}(\alpha) -1\right)^2} < 0,
		\end{equation*}
		which is negative, but not constant, and where $\psi^{(m)}(x) \coloneqq \frac{\d^{m+1}}{\d x^{m+1}}\log \Gamma(x)$ denotes the polygamma function of order~$m$. Bounds for the Fisher--Rao distance in this manifold have been studied in~\cite{burbea2002}.
		
		\item The statistical manifold of Beta distributions $p(x) = \frac{\Gamma(\alpha+\beta)}{\Gamma(\alpha)\Gamma(\beta)} x^{\alpha-1} (1-x)^{\beta-1}$, defined for $x \in [0,1]$ and parametrised by $(\alpha,\beta) \in \R_+^* \times \R_+^*$. The curvature of this manifold is~\cite{lebrigant2021-beta}
		\begin{align*}
			\kappa
			&=
			\frac{ \psi^{(2)}(\alpha) \psi^{(2)}(\beta) \psi^{(2)}(\alpha+\beta) }{4 \left( \psi^{(1)}(\alpha)\psi^{(1)}(\beta) - \psi^{(1)}(\alpha+\beta)[\psi^{(1)}(\alpha) + \psi^{(1)}(\beta)] \right)^2}\\
			&\quad\times \left( \frac{\psi^{(1)}(\alpha)}{\psi^{(2)}(\alpha)} + \frac{\psi^{(1)}(\beta)}{\psi^{(2)}(\beta)} - \frac{\psi^{(1)}(\alpha+\beta)}{\psi^{(2)}(\alpha+\beta)} \right)\\
			&< 0,
		\end{align*}
		which is negative, but not constant too. In fact, more generally, the sectional curvature of the statistical manifold of Dirichlet distributions (which are the multivariate generalisation of Beta distributions) is negative~\cite[Thm.~6]{lebrigant2021-dirichlet}.
		
		\item Finally, a construction of $n$-dimensional statistical manifolds, based on a Hilbert space representation of probability measures, was given in~\cite{minarro1993}. The geometry of these manifolds is spherical, that is, they have constant positive curvature.
	\end{itemize}
\end{remark}

\subsection{Wishart} \label{subsec:wishart}

A Wishart distribution~\cite{ayadi2023,gupta2000} in dimension~$m$, with $n \ge m$ degrees of freedom, $n\in\N$, has p.d.f.
\begin{equation*}
	p(X) = \frac{
		{(\det X )}^{(n-m-1)/2}\exp\left(-\frac{1}{2}\tr(\Sigma^{-1}X)\right)
	}{
		2^{nm/2} (\det \Sigma)^{n/2} \, \Gamma_m(n/2)
	},
\end{equation*}
defined for $X\in P_m(\R)$, and characterised by $\Sigma \in P_m(\R)$, where $\Gamma_m(z) \coloneqq \pi^{\frac{m(m-1)}{4}}\prod_{j=1}^{m} \Gamma\left(z+\frac{1-j}{2}\right)$ denotes the multivariate Gamma function. Note that, for fixed $m,n$, the associated $\left(\frac{m(m+1)}{2}\right)$-dimensional statistical manifold~$\Scal$ is in correspondence with the cone $P_m(\R)$ of symmetric positive-definite matrices via the bijection $\iota \colon P_m(\R) \to \Scal$, given by $\Sigma \mapsto  \frac{{(\det X )}^{(n-m-1)/2}\exp\left(-\frac{1}{2}\tr(\Sigma^{-1}X)\right)}{2^{nm/2} (\det \Sigma)^{n/2} \, \Gamma_m(n/2)}$. Denoting $\sigma_{i,j}$ the $(i,j)$-th entry of the matrix $\Sigma$, we can write the parameter vector as $\xi = \left(\xi^1, \dots, \xi^{m(m+1)/2}\right) = (\sigma_{1,1}, \dots \sigma_{1,m}, \sigma_{2,2}, \dots, \sigma_{2,m}, \dots, \sigma_{m,m})$. We then have $\partial_{i}\ell \coloneqq \partial_{{i}} \ell(\Sigma) = \frac{1}{2}\tr\left(\Sigma^{-1}X\Sigma^{-1} {(\partial_i\Sigma)}\right) - \frac{n}{2}\tr\left(\Sigma^{-1}{(\partial_i\Sigma)}\right)$ and  $\partial_{j}\partial_{i}\ell = - \tr\left(\Sigma^{-1}{(\partial_i\Sigma)}\Sigma^{-1}{(\partial_j\Sigma)}\Sigma^{-1}X\right) + \frac{n}{2}\tr\left(\Sigma^{-1}{(\partial_i\Sigma)}\Sigma^{-1}{(\partial_j\Sigma)}\right)$, where the derivative in $\partial_i \Sigma$ is taken entry-wise. The elements of the Fisher metric are then
\begin{align}
	g_{ij}(\xi)
	&= -\E \left[ \partial_{j}\partial_{i}\ell\right] \nonumber\\
	&= -\E \left[- \tr\left(\Sigma^{-1}{(\partial_i\Sigma)}\Sigma^{-1}{(\partial_j\Sigma)}\Sigma^{-1}X\right) + \frac{n}{2} \tr\left(\Sigma^{-1}{(\partial_i\Sigma})\Sigma^{-1}{(\partial_j\Sigma)}\right) \right] \nonumber\\
	&=\tr\left(\Sigma^{-1}{(\partial_i\Sigma)}\Sigma^{-1}{(\partial_j\Sigma)}\Sigma^{-1}\E[X]\right) - \frac{n}{2}\tr\left(\Sigma^{-1}{(\partial_i\Sigma)}\Sigma^{-1}{(\partial_j\Sigma)}\right) \nonumber\\
	&= \frac{n}{2}\tr\left(\Sigma^{-1}{(\partial_i\Sigma)}\Sigma^{-1}{(\partial_j\Sigma)}\right),
\end{align}
where we have used that $\E[X]=n\Sigma$.

In view of the bijection~$\iota$, the tangent space $T_{p_\xi}\Scal$ can be identified with $H_m$, the set of $m \times m$ symmetric matrices~\cite[Chapter~6]{bhatia2007}. Given two matrices $U, \ \tilde{U} \in H_m$, parametrized as $\theta = \left(\theta ^1, \dots, \theta ^{m(m+1)/2}\right) = \left(u_{1,1}, \dots u_{1,m}, u_{2,2}, \dots, u_{2,m}, \dots, u_{m,m}\right)$ and $\tilde{\theta} = \left(\tilde\theta ^1, \dots, \tilde\theta ^{m(m+1)/2}\right) = (\tilde{u}_{1,1}, \dots \tilde{u}_{1,m}, \tilde{u}_{2,2}, \dots, \tilde{u}_{2,m}, \dots, \tilde{u}_{m,m})$, we shall compute the inner product defined by the Fisher metric, cf.~\eqref{eq:inner-product}. In the following, $\otimes$ denotes the Kronecker product, and, for an $m \times m$ matrix $A$, whose columns are $A_1, A_2, \dots, A_m$, we denote $\vect(A) \coloneqq \begin{pmatrix} A_1^\t & A_2^\t & \cdots &A_m^\t \end{pmatrix}^\t$ the $m^2$-dimensional vector formed by the concatenation of its columns.  Moreover, if $A$ is symmetric, denote $\nu(A) \coloneqq (a_{1,1}, \dots a_{1,m}, a_{2,2}, \dots, a_{2,m}, \dots, a_{m,m})$. We denote $D_{m}$ the unique $m^2 \times \frac{m(m+1)}{2}$ matrix that verifies $D_{m} \nu(A) = \vect(A)$, for any symmetric~$A$~\cite[\S~7]{magnus1986}. We thus have
\begin{align*}
	\langle U, \tilde{U} \rangle_{G(\xi)}
	&= \theta^\t G(\xi) \tilde\theta\\
	&= \sum_{i=1}^{\frac{m(m+1)}{2}} \sum_{j=1}^{\frac{m(m+1)}{2}} g_{ij}(\xi) \theta^i \tilde\theta^j\\
	&= \sum_{i=1}^{\frac{m(m+1)}{2}} \sum_{j=1}^{\frac{m(m+1)}{2}} \frac{n}{2} \tr\left(\Sigma^{-1}{(\partial_i\Sigma)}\Sigma^{-1}{(\partial_j\Sigma)}\right) \theta^i \tilde\theta^j\\
	&= \frac{n}{2} \tr\left(\Sigma^{-1}U\Sigma^{-1} \tilde{U} \right)\\
	&=\frac{n}{2} \vect(U)(\Sigma^{-1}\otimes\Sigma^{-1})\vect(\tilde{U})\\
	&=\frac{n}{2}\theta^\t D_{m}^\t (\Sigma^{-1}\otimes\Sigma^{-1})D_{m}\tilde\theta,
\end{align*}
where we have used that $\tr(ABCD)=\vect(D)^\t (A\otimes C^\t)\vect(B^\t)$, for $A$, $B$, $C$ and $D$ matrices such that the product $ABCD$ is defined and square~\cite[Lemma~3]{magnus1986}. We can thus conclude that the Fisher matrix is 
\begin{equation}\label{eq:fisher-matrix-wishart}
	G(\xi)=\frac{n}{2} D_m^\t \left(\Sigma^{-1}\otimes\Sigma^{-1}\right)D_m.
\end{equation}

This metric turns out to coincide with the Fisher metric of the statistical manifold formed by multivariate Gaussian distributions with fixed mean~\cite{burbea1986,nielsen2023,pinele2020,skovgaard1984}, up to the factor~$n$. Therefore, the Fisher--Rao distance is proportional to the one of that manifold, that is,
\begin{equation} \label{eq:distance-wishart}
	d_{\FR}\big( \Sigma_1,\Sigma_2 \big)
	= \sqrt{\frac{n}{2}} \left\| \log \left( \Sigma_1^{-1/2}\Sigma_2\Sigma_1^{-1/2} \right) \right\|_F
	= \sqrt{\frac{n}{2}\sum_{k=1}^{m} \left(\log \lambda_k \right)^2},
\end{equation}
where $\lambda_k$ are the eigenvalues of $\Sigma_1^{-1/2}\Sigma_2\Sigma_1^{-1/2}$, $\log$ denotes the matrix logarithm, and $\|A\|_F = \sqrt{\tr\left( AA^\t \right)}$ is the Frobenius norm. This metric also coincides with the standard metric of $P_m(\R)$, when endowed with the matrix inner product $\langle A,B \rangle = \tr(A^{\t}B)$~\cite[Chapter~6]{bhatia2007}, up to the factor $\frac{n}{2}$, and is in fact related to the metric of the Siegel upper space~\cite{siegel1943} (see also~\cite[Appendix~D]{nielsen2023}). Note that when $\Sigma$ is restricted to be diagonal this distance is, up to a factor $\sqrt{n}$, the product distance between univariate Gaussian distributions with fixed mean---see Example~\ref{ex:product-gaussian} in Section~\ref{sec:product-distributions}.

\subsection{Inverse Wishart} \label{subsec:inverse-wishart}

An inverse Wishart distribution~\cite{gelman2013} in dimension~$m$, with $n \ge m$ degrees of freedom, $n\in\N$, has p.d.f.
\begin{equation*}
	p(X) = \frac{
		{(\det \Sigma)^{n/2} (\det X )}^{-(n+m+1)/2}\exp\left(-\frac{1}{2}\tr(\Sigma X^{-1})\right)
	}{
		2^{nm/2}\Gamma_m(n/2)
	},
\end{equation*}
defined for $X\in P_m(\R)$, and characterised by $\Sigma \in P_m(\R)$.
We can relate an inverse Wishart distribution to a Wishart distribution by considering the reparametrisation of the sample space given by $Y \coloneqq X^{-1}$ (cf. Proposition~\ref{prop:reparametrisation-sample-space}), and the change of coordinates $\Phi \coloneqq \Sigma^{-1}$. The p.d.f. of the new random variable, in the new coordinates, is 
\begin{align*}
	p(Y) &= \frac{1}{\left| \frac{\d Y}{\d X} \right|} \frac{
		{\left(\det(\Phi^{-1})\right)^{n/2} \left(\det (Y^{-1}) \right)}^{-(n+m+1)/2}\exp\left(-\frac{1}{2}\tr(\Phi^{-1}Y)\right)
	}{
		2^{nm/2} \, \Gamma_m(n/2)
	},\\
	&= \frac{1}{\left(\det(Y^{-1})\right)^{-(m+1)}} \frac{
		{\left(\det (Y^{-1}) \right)}^{-(n+m+1)/2}\exp\left(-\frac{1}{2}\tr(\Phi^{-1}Y)\right)
	}{
		(\det \Phi)^{n/2}2^{nm/2} \, \Gamma_m(n/2)
	}\\
	&=\frac{
		{\left(\det Y \right)}^{(n-m-1)/2}\exp\left(-\frac{1}{2}\tr(\Phi^{-1}Y)\right)
	}{
		(\det \Phi)^{n/2} 2^{nm/2} \, \Gamma_m(n/2)
	},
\end{align*}
where we have used that $\left| \frac{\d Y}{\d X} \right|=\det(X)^{-(m+1)}$~\cite[\S~12]{magnus1986}, which coincides with the p.d.f. of a Wishart distribution with parameter~$\Phi$. Write $\sigma_{i,j}$ and $\phi_{i,j}$ the $(i,j)$-th entries of matrices $\Sigma$ and $\Phi$, respectively. Write $\xi = \left(\xi^1, \dots, \xi^{m(m+1)/2}\right) = (\sigma_{1,1}, \dots \sigma_{1,m}, \sigma_{2,2}, \dots, \sigma_{2,m}, \dots, \sigma_{m,m})$, and $\theta = (\theta^1, \dots, \theta^{m(m+1)/2}) = (\phi_{1,1}, \dots \phi_{1,m}, \phi_{2,2}, \dots, \phi_{2,m}, \dots, \phi_{m,m})$. Denote $G_W(\theta)$ the Fisher matrix of a Wishart distribution~\eqref{eq:fisher-matrix-wishart} in coordinates $\theta$. In the following, $D_m$ is the matrix defined in the previous section, and $A^+$ denotes the Moore-Penrose inverse of a matrix~$A$. Applying Proposition~\ref{prop:reparametrisation-parameter-space} yields
\begin{align*}
	G(\xi) 
	& = \left[\frac{\d\Sigma}{\d\Phi}(\Phi)\right]^{-\t} G_W(\theta) \left[\frac{\d\Sigma}{\d\Phi}(\Phi)\right]^{-1}\\
	& = \frac{n}{2} \left(-D_m^+ \left(\Phi^{-1}\otimes\Phi^{-1}\right) D_m\right)^{-\t} 
	D_m^\t\left(\Phi^{-1}\otimes\Phi^{-1}\right)
	D_m\left(-D_m^+ \left(\Phi^{-1}\otimes\Phi^{-1}\right) D_m\right)^{-1}\\
	& = \frac{n}{2} \left(D_m^+ \left(\Sigma\otimes\Sigma\right)D_m\right)^{-\t}
	D_m^\t\left(\Sigma\otimes\Sigma\right)
	D_m\left(D_m^+ \left(\Sigma\otimes\Sigma\right)D_m\right)^{-1}\\
	& = \frac{n}{2} \left(D_m^+\left(\Sigma^{-1}\otimes\Sigma^{-1}\right)D_m\right)^{\t}
	D_m^\t\left(\Sigma\otimes\Sigma\right)
	D_mD_m^+ \left(\Sigma^{-1}\otimes\Sigma^{-1}\right)D_m\\
	& = \frac{n}{2} D_m^\t\left(\Sigma^{-1}\otimes\Sigma^{-1}\right)
	\left(D_m^+\right)^{\t}
	D_m^\t\left(\Sigma\otimes\Sigma\right)
	D_mD_m^+ \left(\Sigma^{-1}\otimes\Sigma^{-1}\right)D_m\\
	& = \frac{n}{2} D_m^\t\left(\Sigma^{-1}\otimes\Sigma^{-1}\right)D_m,
\end{align*}
where we have used that $\frac{\d\Sigma}{\d\Phi}(\Phi)= -D_m^+\left(\Phi^{-1}\otimes\Phi^{-1}\right)D_m$~\cite[\S~12]{magnus1986}, $\left(D_m^+ (\Sigma\otimes\Sigma)D_m\right)^{-1} = D_m^+ \left(\Sigma^{-1}\otimes\Sigma^{-1}\right) D_m$, and $D_mD_m^+ \left(\Sigma^{-1}\otimes\Sigma^{-1}\right)D_m = \left(\Sigma^{-1}\otimes\Sigma^{-1}\right) D_m$~\cite[Lemma~11]{magnus1986}. Therefore, the Fisher matrix is the same as for Wishart distributions in \S~\eqref{eq:fisher-matrix-wishart}, and the Fisher-Rao distance is
\begin{equation}
	d_{\FR}\big( \Sigma_1,\Sigma_2 \big)
	= \sqrt{\frac{n}{2}} \left\| \log \left( \Sigma_1^{-1/2}\Sigma_2\Sigma_1^{-1/2} \right) \right\|_F
	= \sqrt{\frac{n}{2}\sum_{k=1}^{m} \left(\log \lambda_k \right)^2},
\end{equation}
where $\lambda_k$ are the eigenvalues of $\Sigma_1^{-1/2}\Sigma_2\Sigma_1^{-1/2}$.

\begin{remark}
	Other matrix distributions have been recently studied in the literature.	In~\cite{ayadi2023}, general Wishart elliptical distributions have been addressed, which also include $t$-Wishart and Kotz-Wishart, by noting that their metric coincides with that of zero-mean multivariate elliptical distributions~\cite{berkane1997}.
\end{remark}

\section{Product Distributions} \label{sec:product-distributions}

In this short section, we address the Fisher--Rao distance for multivariate product distributions, i.e., distributions of random vectors whose components are independent. In this case, the distribution of the random vector is the product of the distributions of each component, and the associated statistical manifold is the product of the statistical manifold associated to each component.

Consider $m$ Riemannian manifolds $\{(M_i,g_i) \ \colon \ 1 \le i \le m \}$, and the product manifold $(M,g)$, with $M \coloneqq M_1 \times \cdots \times M_m$, and $g \coloneqq g_1 \oplus \cdots \oplus g_m$. In matrix form, the product metric is given by the block-diagonal matrix
\begin{equation*}
	G = 
	\begin{pmatrix}
		G_1 & 0 & \cdots & 0\\
		0   & G_2 & \cdots & 0\\
		\vdots & \vdots & \ddots & \vdots\\
		0 & 0 & \cdots & G_m
	\end{pmatrix},
\end{equation*}
where $G_i$ is the matrix form of the metric $g_i$, for $1 \le i \le m$. Let $d_i$ denote the geodesic distance in $M_i$. Then, the geodesic distance~$d$ in $(M, g)$ is given by a Pythagorean formula~\cite[Prop.~1]{dandrea2012}, \cite{oller1989}:
\begin{equation}
	d{\left( (x_1, \dots, x_m), (y_1, \dots, y_m) \right)} =  \sqrt{ \sum_{i=1}^{m} \left[d_i(x_i,y_i)\right]^2 },
\end{equation}
where $(x_1, \dots, x_m) \in M_1 \times \cdots \times M_m$, and $(y_1, \dots, y_m) \in M_1 \times \cdots \times M_m$. This result can be used to write explicit forms for the Fisher--Rao distance in statistical manifolds of product distributions, as they are statistical product manifolds. This type of construction has been used, e.g., in \cite{oller1987,villarroya1993}. Some examples are given below.

\begin{example}
	The Fisher--Rao distance between the distributions of $n$-dimensional vectors formed independent Poisson distributions~\cite{burbea1986} (cf.\S~\ref{subsec:poisson}) with parameters $(\lambda_1, \dots, \lambda_n)$ and $(\lambda'_1, \dots, \lambda'_n)$ is
	\begin{equation}
		d_{\FR}\left( (\lambda_1, \dots, \lambda_n), (\lambda'_1, \dots, \lambda'_n) \right)
		= 2 \sqrt{ \sum_{i=1}^{n} \left(\sqrt{\lambda_i} - \sqrt{\lambda'_i}\right)^2 }.
	\end{equation}
\end{example}

\begin{example} \label{ex:product-gaussian}
	Consider multivariate independent Gaussian distributions (cf.~\S~\ref{subsec:gaussian}). In this case, the covariance matrix is diagonal, a case that has been addressed in~\cite{burbea1986,costa2015,pinele2020}. The Fisher--Rao distance between such distributions parametrised by $(\mu_1,\sigma_1, \dots, \mu_n,\sigma_n)$ and $(\mu'_1,\sigma'_1, \dots, \mu'_n,\sigma'_n)$ is
	\begin{align}
		&\hspace{-2em}d_{\FR}\left( (\mu_1, \sigma_1, \dots, \mu_n, \sigma_n), (\mu'_1, \sigma'_1, \dots, \mu'_n, \sigma'_n) \right)\nonumber\\
		&\hspace{2em}=  2\sqrt{2} \sqrt{\sum_{i=1}^{n} \left[\arctanh \left( \sqrt{\frac{(\mu_i-\mu'_i)^2 + 2(\sigma_i-\sigma'_i)^2}{(\mu_i-\mu'_i)^2 + 2(\sigma_i+\sigma'_i)^2}} \right) \right]^2}.
	\end{align}
\end{example}

\begin{example}
	More generally, consider a vector $(X_1,\dots,X_n)$ of $n$-independent generalised Gaussian distributions, where $X_i$ follows a generalised Gaussian with fixed $\beta_i$, parametrised by $(\mu_i,\sigma_i)$. For fixed values $(\beta_1, \dots, \beta_n)$, the distance between the distribution of two such vectors, parametrised by $(\mu_1,\sigma_1, \dots, \mu_n,\sigma_n)$ and $(\mu'_1,\sigma'_1, \dots, \mu'_n,\sigma'_n)$, is
	\begin{align}
		&d_{\FR}\left( (\mu_1, \sigma_1, \dots, \mu_n, \sigma_n), (\mu'_1, \sigma'_1, \dots, \mu'_n, \sigma'_n) \right)\nonumber\\
		&=  \sqrt{
			\sum_{i=1}^{n} (\beta_i+1) \left[\arctanh \left(
			\sqrt{
				\frac{
					\beta_i(\mu_i-\mu'_i)^2\Gamma(2-\frac{1}{\beta_i}) + (\beta_i+1)(\sigma_i-\sigma'_i)^2\Gamma(1+\frac{1}{\beta_i})
				}{
					\beta_i(\mu_i-\mu'_i)^2\Gamma(2-\frac{1}{\beta_i}) + (\beta_i+1)(\sigma'_i+\sigma'_i)^2\Gamma(1+\frac{1}{\beta_i})}
			} 
			\right)\right]^2
		}.
	\end{align}
\end{example}

\section{Final Remarks} \label{sec:conclusion}

In this survey, we have collected closed-form expressions for the Fisher--Rao distance in different statistical manifolds of both discrete and continuous distributions. In curating this collection in a unified language, we also provided some punctual contributions. The results are summarised in Tables~\ref{tab:discrete-distributions}, \ref{tab:continuous-distributions} and \ref{tab:continuous-distributions-2}. We hope that providing these expressions readily available can be helpful not only to those interested in information geometry itself, but also to a broader audience, interested in using these distances in different applications.

\section*{Acknowledgements}

Part of this work was developed while the first author was with IMECC, Unicamp, Brazil. H.K.M. was supported by São Paulo Research Foundation (FAPESP) grant 2021/04516-8. F.C.C.M. was supported by Brazilian National Council for Scientific and Technological Development (CNPq) grant 141407/2020-4. S.I.R.C. was supported by FAPESP grant 2020/09838-0 (BI0S -- Brazilian Institute of Data Science) and CNPq grant 314441/2021-2.

Discussions with Luiz~Lara and Franciele~Silva are gratefully acknowledged. We thank Gabriel~Khan for bringing to our attention the example of inverse Gaussian distributions.

\bibliographystyle{unsrt}
\bibliography{references}

\end{document}